\documentclass[12pt]{amsart}
\usepackage{fullpage}
\usepackage{bm}
\usepackage[all]{xy}
\usepackage{tikz}
\usepackage{amsmath, amsthm, amscd, amssymb, amsfonts, latexsym}
\usepackage{mathrsfs,amssymb}
\usepackage{amsfonts}
\usepackage{amssymb}
\usepackage{amsmath}
\usepackage{amsthm}
\usepackage{bbm}
\usepackage{float}
\usepackage{dsfont,multirow}
\usepackage{relsize}
\usepackage[mathscr]{euscript}

\newtheorem{theorem}{Theorem}[section]
\newtheorem{lemma}[theorem]{Lemma}
\newtheorem{proposition}[theorem]{Proposition}
\newtheorem{corollary}[theorem]{Corollary}

\newtheorem{definition}[theorem]{Definition}

\newtheorem{remark}[theorem]{Remark}

\newcommand{\Q}{\mathbb{Q}}

\newcommand{\Z}{\mathbb{Z}}

\newcommand{\F}{\mathbb{F}}

\newcommand{\C}{\mathbb{C}}

\newcommand{\kommentar}[1]{}

\def\cL{\mathcal{L}}
\def\cZ{\mathcal{Z}}

\newcommand{\degL}{\mathfrak{n}}

\newcommand{\ccom}[1]{{\color{red}{Chantal: #1}} }
\newcommand{\mcom}[1]{{\color{brown}{Matilde: #1}} }

\begin{document}

\title{On the vanishing of twisted $L$-functions of elliptic curves over rational function fields}
\author{Antoine Comeau-Lapointe, Chantal David, Matilde Lalin and Wanlin Li}

\address{Antoine Comeau-Lapointe: Department of Mathematics and Statistics, Concordia University, 1455 de Maisonneuve West, Montr\'eal, Qu\'ebec, Canada H3G 1M8}
\email{antoine.comeau-lapointe@concordia.ca}

\address{Chantal David: Department of Mathematics and Statistics, Concordia University, 1455 de Maisonneuve West, Montr\'eal, Qu\'ebec, Canada H3G 1M8}
\email{chantal.david@concordia.ca}

\address{Matilde Lal\'in:  D\'epartement de math\'ematiques et de statistique,
                                    Universit\'e de Montr\'eal.
                                    CP 6128, succ. Centre-ville.
                                     Montreal, QC H3C 3J7, Canada}\email{mlalin@dms.umontreal.ca}

\address{Wanlin Li: Centre de recherches math\'ematiques,
                                    Universit\'e de Montr\'eal.
                                    CP 6128, succ. Centre-ville.
                                     Montreal, QC H3C 3J7, Canada}\email{liwanlin@crm.umontreal.ca}

\begin{abstract}
We investigate in this paper the vanishing  at $s=1$ of the twisted $L$-functions of elliptic curves $E$ defined over the rational function field $\F_q(t)$ (where $\F_q$ is a finite field of $q$ elements and characteristic  $\geq 5$) for twists by Dirichlet characters of prime order $\ell \geq 3$, from both a theoretical and numerical point of view.  In the case of number fields, it is predicted that such vanishing is a very rare event, and our numerical data seems to indicate that this is also the case over function fields for non-constant curves. For constant curves, we adapt the techniques of \cite{Li-vanishing, Donepudi-Li} who proved vanishing at $s=1/2$ for infinitely many Dirichlet $L$-functions over $\F_q(t)$ based on the existence of one, and we can prove that if there is one $\chi_0$ such that $L(E, \chi_0, 1)=0$, then there are infinitely many. Finally, we provide some examples which show that twisted $L$-functions of constant elliptic curves over $\F_q(t)$ behave differently than the general ones.
\end{abstract}

\date{\today}
\numberwithin{equation}{section}
\maketitle

\section{Introduction} 

Let $E$ be an elliptic curve over $\Q$ with $L$-function $L(E, s) =\sum_{n} a_n n^{-s}$, and $\chi$ be a Dirichlet character. Let 
$L(E, \chi, s) = \sum_{n} a_n \chi(n) n^{-s}$ be the twisted $L$-function.
By the Birch and Swinnerton-Dyer conjecture, the vanishing of $L(E, \chi, s)$ at $s=1$ should be related to the growth of the rank of the Mordell-Weil group of $E$ in the abelian extension of $\Q$ associated to $\chi$. Heuristics based on the distribution of modular symbols and random matrix theory (\cite[Conjecture 1.2]{DFK2}, \cite{MR}) have led to conjectures predicting that the vanishing of the twisted $L$-functions $L(E, \chi, s)$ at $s=1$ is a very  rare event as $\chi$ ranges over characters of prime order $\ell\ge 3$. For instance, it is predicted that there are only finitely many characters $\chi$ of order $\ell > 5$ such that $L(E, \chi, 1)=0$. 
 Mazur and Rubin rephrased this in terms of ``Diophantine Stability", and conjectured that if  $E$ is an elliptic curve over $\Q$ and $K/\Q$ is any real
abelian extension such that $K$ contains only finitely many subfields of
degree $2,3$, or $5$ over $\Q$,  then the group of $K$-rational points $E(K)$ is
finitely generated. They also proved that for each $\ell$ (under some hypotheses that can be shown to hold in certain contexts), there are infinitely many cyclic extensions $K/\Q$ of order $\ell$ such that $E(K)=E(\Q)$ (and then, assuming the Birch and Swinnerton-Dyer conjecture, such that the twisted $L$-functions $L(E, \chi, s)$ associated to the extensions $K/\Q$ do not vanish) \cite{MR-Diophantine-Stability}.

 We remark that the case of vanishing of quadratic twists is very different from the higher order case $\ell \ge 3$ considered in this work, as the $L$-function of $E$ twisted by a quadratic character of conductor $D$ corresponds to the $L$-function of another elliptic curve $E_D$, and for half of the quadratic twists, $L(E, \chi_D, 1)=1$. Goldfeld has conjectured that half of the twists $E_D/\Q$ have rank 0, and half have rank 1 (asymptotically) \cite{Goldfeld}). Furthermore, Gouvea and Mazur \cite{gouvea-mazur}
have shown that the analytic rank of $E_D$ is at least two for $\gg X^{1/2-\epsilon}$ of the quadratic discriminants $|D| \leq X$. It is conjectured that the number of such discriminants $|D| \leq X$ should be asymptotic to $C_E X^{3/4} \log^{b_E}(X)$ \cite{CKRS-frequency}, for some constants $C_E$ and $b_E$ depending on the curve $E$. 
The case of nonabelian extensions $K/\Q$ of degree $d$ with Galois group $S_d$ is also different from the abelian extensions of order $\ell \ge 3$: in recent work, Lemke Oliver and Thorne \cite{LOT} showed that there are infinitely many such extensions where $\mbox{rank}(E(K)) > \mbox{rank}(E(\Q))$, for each $d \geq 2$, and Fornea \cite{Fornea} has shown that for some curves $E/\Q$, the
analytic rank of $E$ increases for a positive proportion of the quintic fields with Galois group $S_5$.


The vanishing (and non-vanishing) of twisted $L$-functions of elliptic curves is closely related to the one-level density, which is the study of low-lying zeroes, or the average analytic rank. 
This was studied over number fields and functions fields, for quadratic and higher order twists. For quadratic twists, it is possible to prove results on the one-level density strong enough to deduce that a positive proportion of twists with even (respectively odd) analytic rank do no vanish (respectively vanish of order 1) at the central critical point \cite{HB, Comeau-Lapointe-JNT}. 
The one-level density, or average rank, of higher order twists for elliptic curves $L$-functions was studied by \cite{cho} over number fields and \cite{Meisner-S,Comeau-Lapointe-JNT} over function fields.
Quadratic twists of elliptic curve over functions fields were also studied by \cite{BFKRG}
who obtained results on the correlation of the analytic ranks of two twisted elliptic curves. The behavior of the algebraic rank of elliptic curves in cyclic extensions of $\Q$ was investigated by Beneish, Kundu, and Ray \cite{BKR}.


We investigate in this article the vanishing  at $s=1$ of the twisted $L$-functions of elliptic curves $E$ defined over the rational function field $\F_q(t)$, \footnote{Throughout this article, we assume that $\F_q$ is a finite field of $q$ elements and characteristic $\geq 5$.} for twists by Dirichlet characters of prime order $\ell \geq 3$,
from both a theoretical and numerical point of view.  It is natural to ask if the recent results of  Li \cite{Li-vanishing} and Donepudi and Li \cite{Donepudi-Li}, who have found infinitely many instances of vanishing for $L$-functions of Dirichlet characters at $s=1/2$, can be extended to $L$-functions of elliptic curves twisted by Dirichlet characters. We find that this is the case when $E$ is a constant elliptic curve over $\F_q(t)$\footnote{Constant elliptic curves, i.e. elliptic curves over $\F_q$ considered as a curve over $\F_q(t)$, were studied by many authors because of their special properties. In particular, Milne showed that the Birch and Swinnerton-Dyer conjecture is true for constant elliptic curves \cite{Milne}.},
and we can produce infinitely many cases of vanishing at the central critical point for characters of order $\ell$ provided we find one
(Theorem  \ref{thm2}).
Then, the conjectures of \cite{DFK2,MR-experimental} do not hold in the special case of constant elliptic curves, and we present specific numerical examples in Section \ref{section-numerical-constant}.

We also study non-constant elliptic curves over $\F_q(t)$ where $q$ is a power of a prime $p \ge 5$, say
$E: y^2 = x^3 + a(t) x + b(t)$, for some polynomials $a(t), b(t) \in \F_q[t]$.
 The $L$-function of $E/\F_q(t)$ is defined analogously as for $E/\Q$, by an infinite Euler  product over the primes of $\F_q(t)$ (see \eqref{L-function}), but in this case, it follows from the work of Weil and Deligne that, after setting $u=q^{-s}$,  $L(E,s)= \mathcal{L}(E, u)$, a polynomial in $\Z[u]$.
 Similarly, the twisted $L$-function $\mathcal{L}(E, \chi, u)$ is a polynomial in $\Z[\zeta_\ell][u]$, where $\chi$ is a Dirichlet character of order $\ell$ over $\F_q(t)$. More details and all relevant definitions are given in Section \ref{background}.  

We present in Section \ref{section-numerical} computational results  for the vanishing of numerous twists of two base elliptic curves over $\F_q(t)$, the Legendre curve and a second curve, chosen to have good reduction at infinity. The data seems to indicate that the conjectures of \cite{DFK2, MR-experimental} also hold for non-constant elliptic curves over function fields, while presenting some unexpected features.
To our knowledge, this is the first  data about the vanishing of $L$-functions of elliptic curves twisted by characters of order $\ell \geq 3$, over function fields. 
The case of quadratic twists of elliptic curves over function fields was considered by Baig and Hall \cite{BaigHall} to test Goldfeld's conjecture in that context,
and our numerical computations are similar. 
  
The case of a constant curve $E/\F_q(t)$  is defined by taking an elliptic curve $E_0/\F_q$ and considering its base change to $\F_q(t)$, denoted by $E = E_0 \times_{\F_q} \F_q(t)$. In this case, the roots of $\cL(E, \chi, u)$ can be described
in terms of the roots of the $L$-functions $\mathcal{L}(E_0, u)$ and $\cL(C, u)$, where the $L$-functions are respectively associated to the elliptic curve $E_0/\F_q$ and the $\ell$-cyclic cover $C$ 
over $\mathbb{P}^1_{\F_q}$ corresponding to the Dirichlet character $\chi$  (see Section \ref{section-constant}).
This allows us to use a generalized version of the results of Li \cite{Li-vanishing} and Donepudi--Li \cite{Donepudi-Li} about vanishing of the Dirichlet $L$-functions $\cL(\chi, u)$ to obtain some vanishing for $\cL(E, \chi, u)$ at $u=q^{-1}$. 
The argument of \cite{Li-vanishing, Donepudi-Li} has two distinct parts, first finding one character $\chi_0$ such that
$\cL(\chi_0, u_0)=0$ for some fixed $u_0$, and then sieving to produce infinitely many such characters. 
The order of $q \bmod \ell$ is related to the presence/absence of $\ell$-th roots of unity in $\F_q(t)$, which makes the study of the characters  of order $\ell$ delicate, and
the authors of \cite{Li-vanishing, Donepudi-Li} restrict to the Kummer case where $q \equiv 1 \bmod \ell$. As we need to treat all the cases (in particular, we often work over the finite field $\F_p$ where $p$ is prime), we generalize their sieving beyond the Kummer case.  We also need to consider vanishing at any $u_0$ where $\cL(E_0,u_0)=0$, and not only $u_0=q^{-1/2}$ as in their work.

We recall that an algebraic integer $\alpha$ is called a $q$-Weil integer if $|\alpha|=q^{1/2}$ under every complex embedding.

\begin{theorem} \label{thm1} \label{prop-general-counting} Let $\ell$ be a prime and $q$ be a prime power coprime to $\ell$. Let $u_0$ be a $q$-Weil integer. Suppose there exists a Dirichlet character $\chi_0$ over $\F_q(t)$ of order $\ell$ and with conductor of degree $d_0$ such that $\cL(\chi_0, u_0^{-1})=0$. 
Then, there are at least $\gg q^{2n/d_0}$ Dirichlet characters $\chi$ of order $\ell$ over $\F_q(t)$ with conductor of degree bounded by $n$ such that  $\cL(\chi, u_0^{-1})=0$.
\end{theorem}

We prove the above theorem in Section \ref{section-sieve}. The next result is then a direct consequence of Theorem \ref{prop-general-counting}, using the properties of constant elliptic curves discussed in Section \ref{section-constant}.

\begin{theorem} \label{thm2} \label{isotrivial-counting}  Let $E_0$ be an elliptic curve over $\F_q$, and let $E = E_0 \times_{\F_q} \F_q(t)$. Suppose there exists a Dirichlet character $\chi_0$ over $\F_q(t)$ of order $\ell$ and with conductor of degree $d_0$ 
 such that $\cL(E, \chi_0, q^{-1})=0$. Then, there are at least $\gg q^{2n/d_0}$ Dirichlet characters $\chi$ of order $\ell$ over $\F_q(t)$ with conductor of degree bounded by $n$ such that  $\cL(E, \chi, q^{-1})=0$.
\end{theorem}

 Then, to guarantee that a constant elliptic curve $E/\F_q(t)$ has infinitely many twists of order $\ell$ such that $L(E, \chi, u)$ vanishes at $q^{-1}$, it suffices to find one. 
 Using the results of Section \ref{section-constant}, this can be rephrased in terms of finding curves $C/\F_q$ which are $\ell$-cyclic covers of $\mathbb{P}^1_{\F_q}$ and such that $\cL(E_0, u)$ divides $\cL(C, u)$, and we investigate this question numerically in Section \ref{section-isotrivial}, where we 
find isogeny classes of elliptic curves $E_0$ over different prime fields such that $\cL(E,\chi,q^{-1})=0$ for characters $\chi$ of prime order $\ell =3,5,7,11$.
 One observation from the data is the existence of supersingular curves defined over primes fields $\F_p$ which admit a degree $\ell$ cyclic map to $\mathbb{P}^1$ ramifying at $4$ points where $p \equiv -1 \bmod \ell$. The existence of such curves does not follow from previous results on the topic and one may hope to prove this statement following the strong evidence presented in Table \ref{table:d4}.
 
 It is natural to ask if the same dichotomy (no instances of vanishing or infinitely many cases of vanishing) also holds for non-constant elliptic curves over $\F_q(t)$, but there is no reason to believe it would be the case.
The ideas leading to the proof of Theorem \ref{isotrivial-counting} for constant curves do not apply to the general case, as the change of variable trick used to produce infinitely many extensions where $E$ acquires points would send points on $E$ to points on a different elliptic curve when $E$ is not constant. However, there are results of that type for an elliptic curve $E$ over $\Q$ due to Fearnley, Kisilevsky, and Kuwata \cite{FKK}, where the authors prove that if there is one cyclic cubic field $K$ such that $E(K)$ is infinite, then there are infinitely many, and there are always infinitely many such $K$ when $E(\Q)$ contains at least 6 points. On the non-vanishing side, Brubaker, Bucur, Chinta, Frechette and Hoffstein \cite{BBCFH} use the method of multiple Dirichlet series to prove that if there exists a single non-vanishing order $\ell$ twist of an $L$--function associated to a cuspidal automorphic representation of $GL(2,\mathbb{A}_K)$, then there are infinitely many.

The structure of this article is as follows: we define in Section \ref{background} the $L$-functions attached to Dirichlet characters and elliptic curves over $\F_q(t)$, and we recall their properties. We discuss in Section \ref{section-constant} the case of  $L$-functions of constant elliptic curves.
We describe the $\ell$-cyclic covers of $\mathbb{P}^1_{\F_q}$  and their characters in Section \ref{section-sieve}, for all cases (not only the Kummer case $q \equiv 1 \bmod \ell$) using the work of Bary-Soroker and Meisner \cite{BSM},
 and we then generalize the sieves of \cite{Li-vanishing, Donepudi-Li} to those general $\ell$-cyclic covers. We then use those results to prove Theorems \ref{thm1} and \ref{thm2}.
Finally, we describe our computations in Section \ref{section-algo}, and we present our numerical data in Sections  \ref{section-numerical-constant} and \ref{section-numerical}.

\noindent {\bf Acknowledgments.} The authors would like to thank Patrick Meisner for helpful discussions, and the anonymous referees for helpful comments that greatly improved the exposition of this paper.
This work is supported by the Natural Sciences and Engineering Research Council of Canada
(NSERC Discovery Grant \texttt{155635-2019} to CD, \texttt{335412-2013} to ML), by the Fonds de recherche du Qu\'ebec - Nature et technologies (Projet de recherche en \'equipe \texttt{300951} to CD and ML), and by the Centre de recherches math\'ematiques and the Institut des sciences math\'ematiques (CRM-ISM postdoctoral fellowship to WL). Some of the computations were checked using the computational software MAGMA.

\section{Dirichlet characters, elliptic curves and $L$-functions over $\F_q(t)$}
\label{background} \label{section-2}

\subsection{Dirichlet characters of order $\ell$} 

Let $\ell$ be a prime not dividing $q$. We review here the theory of Dirichlet characters of order $\ell$ over $\F_q(t)$ and their $L$-functions. We refer the reader to \cite{DFL} and \cite{BSM} for more details.

Let $n_q$ be the multiplicative order of $q$ modulo $\ell$. We say that we are in the Kummer case if $n_q=1$ and in the non-Kummer case otherwise. We also say that a monic irreducible polynomial $P \in \F_q[t]$ is $n_q$-divisible if $n_q \mid \deg{P}$.  

We fix once and for all an isomorphism $\Omega$ from the $\ell$-th roots of unity in $\F_{q^{n_q}}^*$ to $\mu_\ell$, the $\ell$-th roots of unity in $\C^*$.

We first define the $\ell$-th order residue symbol $$\chi_P : \F_q[t]/(P) \rightarrow \mu_\ell,$$ for $P$ an irreducible $n_q$-divisible monic polynomial in $\F_q[t]$. 
It is clear that  the $\ell$-th residue symbols $\chi_P$ can be defined only for the $n_q$-divisible primes $P$, since we must have $\ell \mid q^{\deg{P}}-1$: indeed, unless $n_q\mid \deg(P)$, the order of the group of non-zero elements in the residue field $\F_P = \F_q[t]/(P)$ is not divisible by $\ell$, and therefore it does not contain any non-trivial $\ell$-th root of unity.

For any $a \in \F_q[t]$, if $P \mid a$, then $\chi_P(a) = 0$, and otherwise
$\chi_{P}(a) = \alpha,$
where $\alpha$ is the unique $\ell$-th root of unity in $\mathbb{C}^*$ such that
\begin{align} \label{def-chiP}  a^{\frac{q^{\deg(P)}-1}{\ell}} \equiv \Omega^{-1}(\alpha) \bmod P. \end{align}
If $F \in \F_q[t]$ is any monic polynomial supported only on $n_q$-divisible primes, writing $F = P_1^{e_1} \cdots  P_s^{e_s}$ with distinct primes $P_i$, we define 
\begin{equation*}
\chi_F = \chi_{P_1}^{e_1} \cdots \chi_{P_s}^{e_s}.
\end{equation*}
Then, $\chi_F$ is a character of order dividing $\ell$ with conductor $P_1 \cdots P_s$. Conversely, the primitive characters of order $\ell$ and conductor $P_1 \cdots P_s$, where the $P_i$ are $n_q$-divisible primes, are given by taking all choices  $1 \leq e_i \leq \ell-1$. Then, the conductors of the primitive characters are the square-free monic polynomials $F \in  \F_q[t]$ supported on $n_q$-divisible primes, and for each such conductor, there are $(\ell-1)^{\omega(F)}$ such characters, where $\omega(F)$ is the number of primes dividing $F$. 


We can also write each primitive character of order $\ell$ with conductor $F$ as
\begin{align} \label{decompose} \chi_F = \chi_{F_1} \chi^2_{F_2} \cdots \chi^{\ell-1}_{F_{\ell-1}}\end{align}
corresponding to a decomposition $F=F_1 \cdots F_\ell$ where the $F_i$'s are square-free and coprime.

 For any Dirichlet character $\chi$, we say that $\chi$ is even if its restriction  to $\F_q$ is trivial; otherwise, we say that $\chi$ is odd. 

Dirichlet characters are also defined at the prime at infinity $P_\infty$. The following statement clarifies how to compute $\chi(P_\infty)$. 
\begin{lemma} \label{leeme-section2} Let $F$ be a monic squarefree polynomial in $\F_q[t]$, and $\chi$ be a Dirichlet character on $\F_q[t]$ of order $\ell$ with conductor $F$. 

If $q\not \equiv 1 \bmod{\ell}$, then $\chi$ does not ramify at infinity, $\chi(P_\infty)= 1$, and $\chi$ is even. 

If $q\equiv 1 \bmod{\ell}$, let $\chi = \chi_{F_1} \chi_{F_2}^2\cdots \chi_{F_{\ell-1}}^{\ell-1}$ as in \eqref{decompose}.
Then,\\ $\chi$ ramifies at $P_\infty \iff \ell \nmid \deg(F_1 F_2^2 \cdots F_{\ell-1}^{\ell-1}) \iff \chi$ is odd, and 
$$\chi(P_\infty) = \begin{cases} 1 &  \ell \mid \deg(F_1 F_2^2 \cdots F_{\ell-1}^{\ell-1}), \\
0 & \ell \nmid \deg(F_1 F_2^2 \cdots F_{\ell-1}^{\ell-1}).\end{cases}$$
\end{lemma}
\begin{proof} 
We first discuss under which conditions the characters are odd or even.  
Let $P$ be an $n_q$-divisible prime. 
We remark that for  $a \in \F_q^*$, 
\begin{equation}\label{eq:chalpha}\chi_P(a)=\Omega\left(a^{\frac{q^{\deg (P)}-1}{\ell}}\right)=\Omega\left(a^{\frac{\deg(P)(q^{n_q}-1)}{n_q \ell}}\right).\end{equation}
Indeed, writing $\deg(P)=n_qk$, we have
\[\frac{q^{\deg (P)}-1}{\ell}=\frac{q^{n_q k}-1}{\ell}=\frac{q^{n_q}-1}{\ell}(1+q^{n_q}+\cdots+q^{n_q(k-1)})\]
and we use the fact that $1+q^{n_q}+\cdots+q^{n_q(k-1)} \equiv k \bmod{\ell}$. 

Then by applying multiplicativity to equation \eqref{eq:chalpha}, we find 
\[\chi_F(a)=\Omega\left(a^{\frac{\deg(F_1F_2^2\cdots F_{\ell-1}^{\ell-1})(q^{n_q}-1)}{n_q\ell}}\right),\]
If $n_q=1$, then $\chi$ is trivial on $\F_q$ iff $\ell \mid \deg(F_1F_2^2\cdots F_{\ell-1}^{\ell-1})$. 

Now suppose that $n_q>1$. Then, $\ell \nmid (q-1)$, and in fact,
$(\ell, q-1)=1$ since $\ell$ is prime. Now we have that both $\ell \mid (q^{n_q}-1)$ and $(q-1)\mid (q^{n_q}-1)$. It follows that  $(q-1)\mid \frac{q^{n_q}-1}{\ell}$. Since $a \in \F_q^*$, we have
\[\chi_F(a)=\Omega\left(a^\frac{\deg(F_1F_2^2\cdots F_{\ell-1}^{\ell-1})(q^{n_q}-1)}{n_q\ell}\right)=1,\] and therefore $\chi_F$ is an even character. 

The statement that $P_\infty$ does not  ramify in the non-Kummer case follows from the fact that the cyclic field extension associated to $\chi_F$ can only ramify at primes of degree divisible by $n_q>1$ and $P_\infty$ is a prime of degree $1$.  In the Kummer case, the character $\chi_F$ is associated with the cyclic cover $y^\ell = F_1 F_2^2 \cdots F_\ell^{\ell-1}$, and there is ramification at $P_\infty$ iff $\ell \nmid \deg(F_1F_2^2\cdots F_{\ell-1}^{\ell-1})$, and 
$\chi_F(P_\infty)=0$ in this case. If $\chi_F$ does not ramify at $P_\infty$, then 
 $\chi_F(P_\infty)=1$ since we are only considering the case in which $F_1F_2^2\cdots F_{\ell-1}^{\ell-1}$ is monic. 
\end{proof}

\subsection{$L$-functions of Dirichlet characters}

Let $\chi$ be a Dirichlet character, and let $\mathcal{L}(\chi, u)$ be the Dirichlet $L$-function defined by 
\[\mathcal{L}(\chi,u)=\prod_P (1-\chi(P)u^{\deg P})^{-1},\]
where the product includes the prime at infinity.
 
 We define $\delta_\chi$ by
\begin{equation} \label{def-delta-chi}
\delta_\chi := \begin{cases} 0 &  \mbox{when $\chi$ is even,} \\ 1 &  \mbox{when $\chi$ is odd,} \end{cases} \end{equation}
and we remark from Lemma  \ref{leeme-section2} that $\chi(P_\infty) =1-\delta_\chi$. 

For a primitive character $\chi$ of conductor $F$, it follows from the work of Weil \cite{Weil2} that  $\cL(\chi, u)$ is a polynomial of degree ${\deg(F)-2 + \delta_\chi}$ and 
 satisfies the functional equation 
 \begin{align} \label{FE-Lchi}
\cL(\chi, u) = \omega_\chi \;(\sqrt{q}u)^{\deg(F)-2 + \delta_\chi} \;\cL(\overline{\chi}, 1/(qu)). \end{align}
The sign of the functional equation is
$$
\omega_\chi = \begin{cases} \frac{G(\chi)}{|G(\chi)|}  & \text{when $\chi$ is even}, \\     \\
 \frac{\sqrt{q}}{\tau(\chi)}  \frac{G(\chi)}{|G(\chi)|}  & \text{when $\chi$ is odd},
\end{cases}
$$
where  if 
$\chi$ odd, 
$$
\tau(\chi) = \sum_{a \in \F_q^*} \chi(a) e^{2 \pi i \text{tr}_{\F_q/\F_p}(a)/p},
$$
and for any $\chi$, $G(\chi)$ is the Gauss sum
$$
G(\chi) = \sum_{a \bmod F} \chi(a) e_q\left( \frac{a}{F} \right).$$
Here $e_q$ is the exponential defined by Hayes \cite{hayes} for any $b \in \mathbb{F}_q((1/T))$:
\[e_q(b) = e^{\frac{2 \pi i \text{tr}_{\F_q/\F_p}(b_1) }{p}},\] 
where $b_1$ is the coefficient of $1/T$ in the Laurent expansion of $b$. We refer the reader to \cite{DFL} for a proof of those results.

\subsection{$L$-functions of elliptic curves over $\F_q(t)$}
Let $E$ be an elliptic curve over $\F_q(t)$. Let $P$ be a prime of $\F_q(t)$, i.e $P=P(t) \in \F_q[t]$ is a monic irreducible polynomial or $P=P_\infty$, the prime at infinity. If $P$ is a prime of good reduction, then the reduction of $E$ (which we also denote by $E$) is an elliptic curve over the finite field $\F_P = \F_q[t]/(P) \cong \F_{q^{\deg{P}}}$ (where $\F_\infty \cong \F_q$ since the prime at infinity has degree 1), and
$$\# E(\F_P) = q^{\deg{P}}+1-a_P, \;\; a_P = \alpha_P + {\overline {\alpha}}_P, \;\; |\alpha_P| = \sqrt{q^{\deg P}}.$$
Let
$$\mathcal{L}_P(E, u) := 1 - a_P u + q^{\deg{P}} u^2 = (1 - \alpha_P u) (1 - \overline{\alpha}_P u)$$ be the $L$-function of $E/\F_P$.

If $P$ is a prime of bad reduction, we define
$$
 \cL_P(E, u) = (1-a_P u),
$$
where $a_P = 0, 1, -1$ depending on the type of bad reduction (additive, split multiplicative, and non-split multiplicative respectively). 

Let $N_E$ be the conductor of $E$, which is the product of the primes of bad reduction with the appropriate powers.\footnote{We emphasize that we include the prime at infinity in the conductor of the elliptic curve (if the curve has bad reduction at infinity of course). Our conductor is an effective divisor, written multiplicatively.} 
Let $M_E$ (respectively $A_E$) be the product of the multiplicative (respectively additive) primes of $E$. Then $N_E= M_E A_E^2$. 


The $L$-function of $E$ is defined by
\begin{equation}\label{L-function}
\cL(E,u):= \prod_{P \nmid N_E} \cL_P(E, u^{\deg{P}})^{-1}  \prod_{P\mid N_E} \cL_P(E, u^{\deg{P}})^{-1}.
\end{equation}

It is proven by Weil \cite{Katz02,BaigHall} that $\cL(E,  u)$ is a polynomial of degree\footnote
{The formula for the degree of $\cL(E,  u)$ implies in particular that there are no non-constant elliptic curves over $\F_q(t)$ with conductor of degree smaller than 4, which can be thought of as the analogue to the fact that there are no elliptic curves over $\Q$ with conductor smaller than $11$. }  $\deg{N_E} - 4$ for any non-constant elliptic curve defined over the rational function field
 $\mathbb{F}_q(t)$ and it satisfies the functional equation
\begin{equation}\label{eq:funtceqE}
\cL(E,  u) = \omega_E \;(qu)^{\deg(N_E)-4} \cL(E, 1/(q^2u)),
\end{equation}
where $\omega_E = \pm 1$ is the sign of the functional equation.
We refer the reader to \cite[Appendix]{Brumer}  and \cite{BaigHall} for more details.  

Let $\chi$ be a Dirichlet character of order $\ell$ and conductor $F$, and suppose that $(F, N_E)=1$.  If $\chi$ is odd, we also assume that $E$ has good reduction at $P_\infty$ (since the prime at infinity is not included in the conductor of the Dirichlet character, we need this additional condition to ensure that the places where $\chi$ ramifies and the places of bad reduction for $E$ are disjoint).
The $L$-function of $E$ twisted by $\chi$ is defined by
\begin{align} \label{L-E-chi} \cL(E, \chi,u) &:= \prod_{P\nmid N_E} (1-\chi(P)\alpha_P  u^{\deg(P)})^{-1}(1-\chi(P)\overline{\alpha}_P  u^{\deg(P)})^{-1} \nonumber \\
&\;\;\;\;\; \times \prod_{P\mid N_E} (1-\chi(P)a_P u^{\deg(P)})^{-1}.\end{align}
Let $K$ be the cyclic field extension of degree $\ell$ of $\F_q(t)$ corresponding to $\chi$. Then, 
\begin{align}\label{Artin-conjecture-1}
\cL(E / K,u)  = \cL(E,u)\prod_{i=1}^{\ell-1}\cL(E,\chi^i,u). \end{align}
It follows from the Riemann Hypothesis that
$$
\cL(E / K, u) =  \prod_{j=1}^{B}  \left(1-qe^{i\theta_{j}}u\right).
$$
Since $(F_\chi,N_E)=1$ and $E$ has good reduction at $P_\infty$ when $\chi$ is odd,  \eqref{Artin-conjecture-1} and Theorem \ref{thm-sign-FE} (stated and proven below)  imply that $B=\ell ( \deg{N_E} - 4) + 2 (\ell-1)
(\deg{F} +  \delta_\chi).$

It is well-known that $\cL(E, \chi, u)$ satisfies a functional equation from the work of Weil \cite{Weil2}. The explicit formula for the sign of the functional equation is contained in \cite{Weil2} in a very general context, but we need a precise formula for the numerical computations, so we deduce it below from the work of Tan and Rockmore \cite{Tan, TanRockmore}.

\begin{theorem} \label{thm-sign-FE} Let $\ell$ be a prime, $\chi$ a primitive Dirichlet character of conductor $F$ and order $\ell$, and let $E$ be a non-constant elliptic curve with conductor $N_E$ such that $(N_E, F)=1$.
If $P_\infty \mid N_E$, we also assume that $\chi$ is even.
The $L$-function $\cL(E,\chi,u)$ is a polynomial of degree
\begin{equation*}
\degL := {\deg{N_E} + 2 \deg{F} - 4 + 2 \delta_\chi},
\end{equation*}
where $\delta_\chi$ is given by \eqref{def-delta-chi}. Each $\cL(E,\chi,u)$
satisfies the functional equation
\begin{equation} \label{FE-E-chi}
\cL(E, \chi, u) = \omega_{E \otimes \chi} \; (qu)^{\degL} \; \cL(E, \overline{\chi}, 1/(q^2u)),
\end{equation}
where $\omega_{E \otimes \chi}$ is the sign of the functional equation for 
$\cL(E, \chi, u)$, given by 
$$\omega_{E \otimes \chi} =  \omega_\chi^2 \, \omega_E \,\chi(N_E).$$
\end{theorem}
\begin{proof}
The sign of the functional equation (and the functional equation itself) can be deduced from the modularity of elliptic curves over function fields.
We follow \cite{Tan, TanRockmore} who use modular symbols over function fields. They consider different normalizations, so we explain here how to adjust their work to get the result that we need. Let $K=\F_q(t)$. 
For any place $v$, let $\mathcal{O}_v$ be the associated ring of integers.  If $N=\sum_v N_v v$ is an effective divisor over $K$, let 
\[\Gamma_0(N)=\left \{\left(\begin{array}{cc}a & b\\c & d \end{array}\right)=\left(\left(\begin{array}{cc}a_v & b_v\\c_v & d_v \end{array}\right)\right)_v \in \prod_v \mathrm{GL}_2(\mathcal{O}_v) : c \equiv 0 \bmod{N} \right\}.\]

Let $\mathbb{A}_K$ be the ring of adeles over $K$. Then $\mathbb{A}_K^*$ embeds in $\mathrm{GL}_2(\mathbb{A}_K)$ as diagonal matrices. Also  $\mathrm{GL}_2(K)$ embeds in  $\mathrm{GL}_2(\mathbb{A}_K)$  by the diagonal map.

A $\C$-valued function on $\mathrm{GL}_2(\mathbb{A}_K)$ is called a modular function of level $N$ if it satisfies that $f(\gamma \tau \kappa)=f(\tau)$ for all $\tau \in \mathrm{GL}_2(\mathbb{A}_K)$, $\gamma \in \mathrm{GL}_2(K)$, and $\kappa \in \mathbb{A}_K^* \cdot\Gamma_0(N)$. It is a fundamental result that if $E$ is a non-constant elliptic curve over $K$, then there is a normalized cuspidal
modular function $f$ of level $N_E$ such that the $L$-function of $E$ is the $L$-function of $f$. This also holds for the twisted $L$-functions. To make that statement precise, and use it to get the functional equation, we will follow the notation of \cite{Tan, TanRockmore}, where the $L$-functions are normalized differently (and we will go back to our $L$-function at the end).
Let $f$ be the normalized cuspidal modular function corresponding to $E$, $\chi$ a Dirichlet character of conductor coprime to $N_E$  and we define as \cite[(1.10)]{Tan}
$$
L_f(\chi, s) = \sum_{M} \frac{ c_f(M) \chi(M)}{|M|^{s-1}},
$$
where $M$ runs through all effective divisors, $\chi$ is naturally extended over effective divisors, and the $c_f(M)$ are the normalized coefficients obtained from the Fourier expansion of $f$. 
This is also true when $\chi$ is a quasi-character, which for our purposes is the product of  a Dirichlet character and a map $\chi_s$ given by 
$
\chi_s (M) = |M|^{-s}.$

We now use the modular symbols $\Theta_{f,D}$ to get the functional equation. The modular symbols $\Theta_{f,D}$ are elements of the group ring $R[W_D]$, where 
$W_D = K^* \backslash \mathbb{A}_K^*/U_D$ is the Weil group of a divisor $D$ of $K$, and $R$ is a ring containing all the Fourier coefficients of $f$.
We refer to \cite{Tan} for all the relevant definitions. The modular symbols are used to interpolate special values of the twisted $L$-functions, and we have \cite[Proposition 2]{Tan}, 
\begin{equation} \label{eq:prop2} L_f(\chi,1)=\tau_\chi^{-1} \chi(\Theta_{f,D}),\end{equation} where $\tau_\chi$ is a Gauss sum.
Using quasi-characters, we also have
\begin{equation}\label{eq:prop2-gen}
L_f(\chi,s)=L_f(\chi \chi_{s-1},1)=\tau_{\chi\chi_{s-1}}^{-1} \; (\chi \chi_{s-1})(\Theta_{f,D}).
\end{equation}
 Using the Atkin--Lehner involution $w_{N_E}$, we have when $(D, N_E)=1$ (including at $P_\infty$) \cite[Proposition 3]{Tan}
\begin{equation} \label{eq:prop3}
\Theta_{f,D} = \Theta^{t}_{w_{N_E}(f),D} \; N_E,
\end{equation}
where $t$ is the involution on $R[W_D]$ sending $\sum_{w\in W_D} a_w w$ to $\sum_{w\in W_D} a_w w^{-1}$. 

Applying a quasi-character $\chi$ to  $\Theta=\sum_{w\in W_D} a_w w$ results in $\chi(\Theta)=\sum_{w\in W_D} a_w \chi(w)$, while applying  $\chi$ together with the involution $t$ results in $\chi(\Theta^t)=\sum_{w\in W_D} a_w \chi^{-1}(w)=\chi^{-1}(\Theta)$.

\kommentar{Proposition 3 from Tan is 
\begin{equation}\label{eq:prop3}
\Theta_{f,D}=\Theta_{w_Nf,D}^t N.
\end{equation}
[[[[I've spent a lot of time trying to understand the above, I think the point is (from the proof of Proposition 3, last identity in the proof):
\[\Theta_{w_Nf,D}^t=\sum_w w_Nf\left(\left(\begin{array}{cc}dw & w\\0&1 \end{array} \right)\right) w^{-1}=\sum_w f\left(\left(\begin{array}{cc}dNw^{-1} & Nw^{-1}\\0&1 \end{array} \right)\right) (Nw^{-1}) N^{-1}=\Theta_{f,D} N^{-1}\]

Now notice that from the proof of Proposition 2 (top of page 303), applying the operation $t$ changes $\sum \Psi(w) \chi(w)$ to $\sum \Psi(w) \chi(w^{-1})$, and eventually this goes from $\tau_\chi \chi(M)$ to $\tau_{\chi^{-1}} \chi^{-1}(M)$ so clearly this changes the quasicharacter to $\chi^{-1}$. 
    ]]]]
    }
We apply $\chi \chi_{s-1}$ to \eqref{eq:prop3}, and we combine it with \eqref{eq:prop2-gen} to get
\begin{align*}
L_f(\chi,s) =& \tau_{\chi\chi_{s-1}}^{-1} (\chi \chi_{s-1}) (\Theta_{f,D})\\=&\tau_{\chi\chi_{s-1}}^{-1}\;  (\chi \chi_{s-1}) (\Theta_{w_{N_E}(f),D}^t) \; \chi(N_E ) |N_E|^{-(s-1)}\\
=&\frac{\tau_{\chi^{-1}\chi_{1-s}}}{\tau_{\chi\chi_{s-1}}} L_{w_{N_E} (f)}(\chi^{-1}\chi_{1-s},1)\chi(N_E)  |N_E|^{-(s-1)}\\
=&\frac{\tau_{\chi^{-1}\chi_{1-s}}}{\tau_{\chi\chi_{s-1}}} L_{w_{N_E}( f)}( \chi^{-1},2-s)\chi(N_E)  |N_E|^{-(s-1)}.
\end{align*}
The third line above follows from using \eqref{eq:prop2} with $f$ replaced by $w_{N_E} (f)$ and $\chi\chi_{s-1}$ replaced by $(\chi\chi_{s-1})^{-1}$, together with the observation that the involution $t$ has the effect of inverting the character.   
Using the fact that $f$ is an eigenvector for the self-dual Atkin--Lehner operator, we have $w_{N_E}(f )=  \omega_E f$, where $\omega_E = \pm 1$ is the sign of the functional equation \eqref{eq:funtceqE}, and then $L_{w_{N_E}( f)}( \chi^{-1},2-s) = \omega_E L_f ( \chi^{-1},2-s)$.

To compute the Gauss sums associated with the quasi-characters, we use \cite[(2.2.3)]{TanRockmore}
\begin{equation*}
\tau_{\chi \chi_s} = q^{s (\deg{D}-2)} \tau_\chi,
\end{equation*}
where $\tau_\chi$ is the  Gauss sum of the Dirichlet character $\chi$ of conductor $D$. Replacing above, this gives
\begin{equation}\label{theabove}
\tau_\chi L_f(\chi,s)= \omega_E \tau_{\chi^{-1}} \chi(N_E) q^{(1-s)(\deg(N_E)+2\deg(D)-4)} L_{f}( \chi^{-1},2-s),
\end{equation}
where \cite[(3.4)]{Tan} is a particular case (for $s=1$). 
The twisted $L$-function of the elliptic curve is given by 
$$L(E, \chi, s) = \sum_{M} \frac{c_f(M)|M| \chi(M) }{|M|^s} = \cL(E, \chi, u)$$
for $u = q^{-s}$. The functional equation can be obtained by noticing that 
$L_f(\chi, s) = L(E, \chi, s)$, and replacing in \eqref{theabove}. This leads to
$$
\tau_\chi L(E, \chi,s)= \omega_E  \tau_{\chi^{-1}} \chi(N_E) q^{(1-s)(\deg(N_E)+2\deg(D)-4)} L(E, \chi^{-1},2-s).
$$
Using $u=q^{-s}$, we finally get
\begin{equation} \label{Tan-FE}
\cL(E, \chi,u)= \omega_{E \otimes \chi} (qu)^{(\deg(N_E)+2\deg(D)-4)} \cL(E, \chi^{-1}, 1/(q^2u)),
\end{equation}
where 
\begin{equation*} 
\omega_{E \otimes \chi} = \left( \frac{\overline{\tau_\chi}}{|D|^{1/2}} \right)^2 \omega_E \chi(N_E) .
\end{equation*}
In order to get exactly the statement of the theorem, we need to take into account the difference of notation between \cite{Tan} and this paper.
When $\chi$ is odd and there is ramification at $P_\infty$,  the conductor $D$ of \eqref{Tan-FE} is $P_\infty D'$, where $D' \in \F_q[t]$, and so $D'$ is the definition of the conductor in this paper. Adjusting the formula to make it compatible with our notation, we get for all cases
$$\cL(E, \chi,u)= \omega_{E \otimes \chi} (qu)^{(\deg(N_E)+2\deg(D)-4 + 2 \delta_\chi )} \cL(E, \chi^{-1}, q^2 u^{-1}) ,$$
which is the functional equation \eqref{FE-E-chi}.
Finally, we remark that $\frac{\overline{\tau_\chi}}{|D|^{1/2}}$ is by definition the sign of the functional equation of $\cL(\chi, u)$, since it is the product of the same local Gauss sums because $(D, N_E)=1$,
and we have $\omega_{E \otimes \chi} = \omega_\chi^2 \omega_E \chi(N_E)$.

\end{proof}

\begin{remark} When $E$ is a constant elliptic curve, we prove in the next section that $\cL(E, \chi, u)$ satisfies the same functional equation with $\degL=2 \deg{F} - 4 + 2 \delta_\chi$ and 
$\omega_{E \otimes \chi} = \omega_\chi^2$. This is consistent with the fact that such $E$ has good reduction at all primes of $K$, and therefore $N_E=0$. \end{remark}

\section{$L$-functions of constant elliptic curves over $\F_q(t)$} \label{section-constant}

By class field theory, Dirichlet characters of order $\ell$ over $\F_q(t)$ correspond to cyclic extensions $K/\F_q(t)$ of order $\ell$, where $K = \F_q(C)$ is the function field of a projective smooth curve $C$ defined over $\F_q$.
We call such a curve a $\ell$-cyclic cover of $\mathbb{P}^1_{\F_q}$, or  simply a $\ell$-cyclic cover.

Let $C$ be a 
$\ell$-cyclic cover of $\mathbb{P}^1_{\F_q}$ of genus $g$, and let $K = \F_q(C)$ be the corresponding extension of $\F_q(t)$. The zeta function of $C$ can be expressed as
\begin{align} \label{eq:betas}
\cZ(C, u) &= \cZ(u) \mathcal{L}(C, u) =
\frac{\displaystyle \prod_{j=1}^{2g} (1 - \beta_j u)}{(1-u)(1-qu)}, 
\end{align}
 where $|\beta_j| = q^{1/2}$ for $1 \leq j \leq 2g$, and
 \[\cZ(u)=\frac{1}{(1-u)(1-qu)}.\]
 We also have 
 \begin{align*}
  \mathcal{L}(C, u) = \prod_{i=1}^{\ell-1} \cL (\chi^i, u), 
 \end{align*}
 where the $\chi^i$ are the characters of order $\ell$ associated to the extension $K/\F_q(t)$.

Let $E_0$ be an elliptic curve over $\F_q$ with $L$-function
$$\mathcal{L}(E_0, u) =  (1 - \alpha_1 u) (1 - {\alpha}_2 u) .$$

\begin{theorem} \label{thm-ulmer}
Let $E = E_0 \times_{\F_q} \F_q(t)$, and let $C, K$ and $\alpha_1, \alpha_2$, and the $\beta_j$'s be as above. Then, 
\begin{align}\label{L-function-over-k(C)-1} 
\mathcal{L}(E/K, u) &=   \cZ(C, \alpha_1 u) \cZ(C, \alpha_2 u) \\ \nonumber
& =\frac{\displaystyle \prod_{\substack{1\leq i \leq 2\\ 1 \leq j \leq 2 g}} (1 -  \alpha_i \beta_j u)}{\displaystyle \prod_{1 \leq i \leq 2} (1-\alpha_i u)(1- \alpha_i q u)}.
\end{align}
Moreover, 
$\mathcal{L}(E, \chi, u) = \cL(\chi, \alpha_1 u) \cL(\chi, \alpha_2 u),$
and writing
\begin{align*} 
\mathcal{L}(\chi, u) =  \prod_{1 \leq j \leq 2 g/(\ell-1)} (1 - \gamma_j u),
\end{align*}
then
$$
\mathcal{L}(E, \chi, u) =    \prod_{\substack{1\leq i \leq 2\\ 1 \leq j \leq 2 g/(\ell-1)}} (1 -  \alpha_i \gamma_j u).
$$
\end{theorem}

\begin{proof}
We refer the reader to \cite[Section 3]{Milne} and to \cite[Section 3.2]{Oesterle} for the general proof. To illustrate the ideas, we prove \eqref{L-function-over-k(C)-1} when $K = \F_q(t)$. Since
$\# E_0(\F_{q^n})=q^n+1-\alpha_1^n-\alpha_2^n$, if $P$ is a prime, then
\[\# E(\F_P)=\# E_0(\F_P) = q^{\deg(P)}+1 -\alpha_1^{\deg(P)}-\alpha_2^{\deg(P)}.\]
Since all the primes are of good reduction, we have
\begin{align*}
\cL(E/\F_q(t),u) = \cL(E, u) =&\prod_{P} \big(1-(\alpha_1^{\deg(P)}+\alpha_2^{\deg(P)}) u^{\deg(P)}+q^{\deg(P)}u^{2\deg(P)} \big)^{-1}\\
=&\prod_{P} \big(1-\alpha_1^{\deg(P)}u^{\deg(P)} \big)^{-1} \big(1-\alpha_2^{\deg(P)}u^{\deg(P)}\big)^{-1}\\
=& \frac{1}{(1-\alpha_1u)(1-q\alpha_1u) (1-\alpha_2 u)(1-q\alpha_2 u)} \\ =& \cZ(\alpha_1 u)  \cZ(\alpha_2 u) .\qedhere
\end{align*}
 \end{proof}
 
 \begin{remark} From the above result, it is easy to get the functional equation for  $\cL(E, \chi, u)$ when $E$ is a constant curve, using the functional equation of $\cL(\chi, u)$ given by \eqref{FE-Lchi}.
 Let $m = \deg_{u} \cL(\chi, u) =  2g/(\ell-1).$
 In the notation of Section \ref{section-2}, we have $m = 2g/(\ell-1) =  \deg F - 2 +  \delta_\chi$, and
 \begin{align*}
 \cL(E, \chi, u) &= \cL(\chi, \alpha_1 u) \cL(\chi, \alpha_2 u) =  \omega_\chi \, (\sqrt{q} \alpha_1 u)^m \, \cL(\overline{\chi}, 1/q \alpha_1 u) \, \omega_\chi (\sqrt{q} \alpha_2 u)^m \, \cL(\overline{\chi}, 1/ q\alpha_2 u)\\
 &= \omega_\chi^2 ( q^2  u^2)^m \cL(\overline{\chi}, \alpha_2/(q^2 u))  \cL(\overline{\chi}, \alpha_1 / (q^2 u)) \\
  &= \omega_\chi^2 \, (q u)^{2m}  \, \cL(E, \overline{\chi}, 1/(q^2u)) = \omega_\chi^2 \, (q u)^{ 2\deg F - 4 +  2\delta_\chi}  \, \cL(E, \overline{\chi}, 1/(q^2u)) 
 \end{align*}
 \end{remark}
 
 \begin{corollary} \label{coro:Li} 
 Let $E = E_0 \times_{\F_q} \F_q(t)$, and let $\chi$ be a Dirichlet character over $\F_q(t)$ with associated curve $C$ and function field $K=\F_q(C)$ respectively. Then, 
 $\cL(E/K, q^{-1})=0$ if and only if $\cL(C, \alpha_1^{-1})=\cL(C, \alpha_2^{-1})=0$,
 \end{corollary}
 
 \begin{proof} From equation \eqref{L-function-over-k(C)-1} in Theorem \ref{thm-ulmer}, $\cL(E/K, q^{-1})=0$ if and only if there is one $\beta_j=q / \alpha_1 = \alpha_2$ or $\beta_j=q / \alpha_2 = \alpha_1$, where the $\beta_j$'s are given by \eqref{eq:betas}, and both $\alpha_1^{-1}$ and $\alpha_2^{-1}$ are roots of $\cL(C, u)$, because of the functional equation of $\cL(C,u)$. 
  \end{proof}

\section{Cyclic extensions of degree $\ell$ over $\F_q(t)$} \label{section-sieve}

We prove in this section the following result which extends the result of \cite{Donepudi-Li}  to general $q$ and $\ell$ (removing the restrictions
$q \equiv 1 \bmod \ell$ and  $y^\ell = F(t)$ with $\ell \mid \deg{F}$).

\begin{proposition} \label{general-counting} \label{prop-4.1} Let $\ell$ be an odd prime.
Fix an $\ell$-cyclic cover $C_0$ over $\mathbb{P}^1_{\F_q}$ with conductor of degree $d_0$.
Then there are at least $\gg q^{2n/d_0}$ $\ell$-cyclic covers $C$ over $\mathbb{P}^1_{\F_q}$ with conductor of degree bounded by $n$ admitting a non-constant map from $C$ to $C_0$.
\end{proposition}

The proof of this result is fairly long and will require several intermediate steps. 

\subsection{General $\ell$-cyclic covers over $\mathbb{P}^1_{\F_q}$}

The affine equations of $\ell$-cyclic covers over  $\mathbb{P}^1_{\F_q}$ are well-known in the Kummer case  $q \equiv 1 \bmod \ell$, which is the case  treated in \cite{Donepudi-Li}. 
In this case, such a cover $C$  over  $\mathbb{P}^1_{\F_q}$ has an affine equation
$y^\ell = F_1 F_2^2  \cdots F_{\ell-1}^{\ell-1}$, where $F_i \in \F_q[t]$ are square-free and pairwise co-prime of degree $d_i$. The conductor of the $\ell$-cyclic cover is {$ F_1 \cdots F_{\ell-1}$} and by the Riemann--Hurwitz formula, the genus of $C$ is $\frac{\ell-1}{2}(d_1+\cdots +d_{\ell-1}-2)$ if $\ell \mid (d_1 +2d_2 + \cdots + (\ell-1)d_{\ell-1})$ and $\frac{\ell-1}{2}(d_1+\cdots +d_{\ell-1}-1)$ otherwise. In this later case,   there is ramification at infinity since $\ell \nmid (d_1 +2d_2 + \cdots + (\ell-1)d_{\ell-1})$ by Lemma \ref{leeme-section2}.

To treat the general case and prove Proposition \ref{general-counting}, we use the work of Bary-Soroker and Meisner \cite{BSM}, who explicitly give the affine equations of general $\ell$-cyclic covers over  $\mathbb{P}^1_{\F_q}$. We summarize their results in this section.

As before, let $n_q$ be the multiplicative order of $q$ modulo $\ell$. As seen in Section \ref{background}, the conductors of the $\ell$-cyclic covers of  $\mathbb{P}^1_{\F_q}$ (or of Dirichlet characters of order $\ell$) are monic square-free polynomials in $\F_q[t]$ supported on $n_q$-divisible primes.  In order to count all the $\ell$-cyclic covers, or characters of order $\ell$, with such conductors, let
\begin{align*}
\mathcal{F}_{q,\ell} :=& \{ F \in \F_q[t]\;:\; F=P_1^{e_1} \cdots P_s^{e_s}, \; n_q \mid \deg{P_i}, \; 1 \leq e_i \leq \ell-1 \}, 
\end{align*}
where the $P_i$ are monic irreducible $n_q$-divisible polynomials in $\F_q[t]$.

Let $\phi_q$ be the Frobenius automorphism of $\F_q$. Then, $\phi_q$ acts on $f(t) \in \F_{q^{n_q}}[t]$ by acting on the coefficients, and we define
$$N_{n_q}(f) :=f \phi_q(f) \phi_q^2(f) \cdots \phi_{q}^{n_q-1}(f) \in \F_q[t].$$
Notice that $N_{n_q}(f)$ has degree $n_q\deg(f)$, which is always divisible by $n_q$. 

By hypothesis, each prime $P_i$ in the factorization of $F \in \mathcal{F}_{q,\ell}$ splits as a product of $n_q$ primes in $\F_{q^{n_q}}[t]$, and we can write any $F \in \mathcal{F}_{q,\ell}$ as
\begin{align} \label{decomposition}
F = \mathfrak{F}_{1}\cdots \mathfrak{F}_{n_q}, \;\; \mathfrak{F}_i \in \F_{q^{n_q}}[t], \; \phi_q(\mathfrak{F}_i)=\mathfrak{F}_{i+1} \; 1 \leq i \leq n_q-1, \; \phi_q(\mathfrak{F}_{n_q})=\mathfrak{F}_1.\end{align}

In other words, for $F \in \mathcal{F}_{q,\ell}$, $F = N_{n_q}(\mathfrak{F}_i)$ for any $i$. 
Since $\mathfrak{F}_1$ determines $\mathfrak{F}_i$ for all $i$, it suffices to work with $\mathfrak{F}_1$. Let 
\[\mathcal{F}^{(1)}_{q,\ell} = \{\mathfrak{F}_1 \in \F_{q^{n_q}}[t]: N_{n_q}(\mathfrak{F}_1)\in \mathcal{F}_{q,\ell}\}.\]
Thus,  $\mathfrak{F}_{1} \in \mathcal{F}^{(1)}_{q,\ell}$ when $F \in \mathcal{F}_{q, \ell}$. We also  have
\begin{equation}
\label{product-of-f}
\mathfrak{F}_{1} = f_1 f_2^2 \cdots f_{\ell-1}^{\ell-1},
\end{equation} where the $f_i \in \F_{q^{n_q}}[t]$ are pairwise co-prime and square-free. 

For any vector $\mathbf{v} = (v_1, \dots, v_{n_q}) \in \mathcal{V} = \{0,1,2, \dots, \ell-1 \}^{n_q}$, and any $F \in \mathcal{F}_{q,\ell}$ written as in \eqref{decomposition},  let 
$F_{\mathbf v} =\mathfrak{F}_{1}^{v_1} \cdots \mathfrak{F}_{n_q}^{v_{n_q}}$. 
For $0 \leq k \leq n_q-1$, let $\mathbf{v}_k = ([q^k]_\ell, [q^{k-1}]_\ell, \dots, [q^{k+1-n_q}]_\ell)$, where $[\alpha]_\ell \equiv \alpha \pmod{\ell}$ and $0\leq [\alpha]_\ell \leq \ell-1$, in other words,  $[\alpha]_\ell$ indicates the reduction modulo $\ell$ of $\alpha$. Thus,  we have $\mathbf{v}_k  \in \mathcal{V}$. Let $\zeta_\ell \in \F_{q^{n_q}}$ be a fixed primitive $\ell$th root of unity.
For any $F  \in \mathcal{F}_{q,\ell}$, 
let $C_{F}$ be the curve over $\F_q$ with affine model
\begin{align} \label{BSM}
C_F \;:\; \prod_{j=0}^{\ell-1} \left( y - \sum_{k=0}^{n_q-1} \zeta_\ell^{j q^k} \sqrt[\ell]{F_{\mathbf{v}_k}} \right) = 0.\end{align}
Notice that there is no canonical choice for $\sqrt[\ell]{F_{\mathbf{v}_k}}$, but the above equation is still well defined, since the factors include all the Galois conjugates.

In the Kummer case $n_q = 1$,  $F_{{\mathbf v}_0} = \mathfrak{F}_{1}=F$, and $C_F$ has affine model $y^\ell = F(t)$. 
In the case $\ell=3$ and $q \equiv 2 \bmod 3$, $F=\mathfrak{F}_{1} \mathfrak{F}_{2}$ and by \eqref{BSM}, $C_F$ has equation 
\begin{align*}
C_F : & \; \left(y- \sqrt[3]{\mathfrak{F}_{1} \mathfrak{F}_{2}^2}-\sqrt[3]{\mathfrak{F}_{1}^2 \mathfrak{F}_{2}}\right)  \;
\left(y-\zeta_3 \sqrt[3]{\mathfrak{F}_{1} \mathfrak{F}_{2}^2}-\zeta_3^2\sqrt[3]{\mathfrak{F}_{1}^2 \mathfrak{F}_{2}}\right) \\
 & \times \; \left(y-\zeta_3^2 \sqrt[3]{\mathfrak{F}_{1} \mathfrak{F}_{2}^2}-\zeta_3\sqrt[3]{\mathfrak{F}_{1}^2 \mathfrak{F}_{2}}\right) =0 \\
 \iff & y^3 - 3 \mathfrak{F}_{1}\mathfrak{F}_{2}y - \mathfrak{F}_{1} \mathfrak{F}_{2} (\mathfrak{F}_{1} + \mathfrak{F}_{2}) = 0,
\end{align*}
which is defined over $\F_q$. In general, $C_F$ is birationally equivalent to $y^\ell=F_{{\mathbf v}_0}$ over $\overline{\F}_q$. 
More explicit versions of  \eqref{BSM}  are given in Section \ref{explicit-BSM},  including a precise formula for the case $n_q=2$.

\begin{proposition}{\cite[Proposition 2.14]{BSM}}\label{Prop:monic} Let $B=\{b \in \F_{q^{n_q}}^*/  (\F_{q^{n_q}}^*)^\ell\}$. 
There is a $(\ell-1)$-to-$1$ correspondence between $\mathcal{F}_{q,\ell}\times B$ and the $\ell$-cyclic covers of $\F_q(t)$, and then a $1$-to-$1$ correspondence between $\mathcal{F}_{q,\ell}\times B$ and the characters of order $\ell$ over $\F_q(t)$.
\end{proposition}

We restrict in this paper to characters with monic conductors, and it then suffices to work with the set $\mathcal{F}_{q,\ell}$.


\begin{lemma}\label{lem:dividibility}
	With notation as above, assume $n_q>1$. Then for each $0 \le k \le n_q-1$, we have $\ell \mid \deg( F_{{\bf v}_k}).$
\end{lemma}

\begin{proof}
By construction, 
\begin{align*}
\deg( F_{{\bf v}_k})=& \sum_{j=1}^{n_q} {\bf v}_{k,j} \deg(\mathfrak{F}_j)= \sum_{j=1}^{n_q} {\bf v}_{k,j} \deg\left(\phi^{j-1}\left(f_1f_2^2\cdots f_{\ell-1}^{\ell-1}\right)\right)\\
\equiv& \sum_{j=1}^{n_q} q^{k+1-j} \sum_{h=1}^{\ell-1} h \deg(\phi^{j-1}( f_h))\equiv \sum_{h=1}^{\ell-1} h \deg(f_h)\sum_{j=1}^{n_q} q^{k+1-j}  \bmod{\ell}.
\end{align*}
Since $n_q>1$, 
\[\sum_{j=1}^{n_q} q^{k+1-j} =\frac{q^{k+1-n_q}(q^{n_q}-1)}{q-1} \equiv 0 \bmod{\ell}.\qedhere \]
\end{proof}

\subsection{From one to  infinitely many $\ell$-cyclic covers}

Given an $\ell$-cyclic cover $C_0$, we can build $\ell$-cyclic covers $C$ with a non-constant map to $C_0$ by a change of variables, as done in {\cite[Lemma 3.2]{Donepudi-Li}} for the Kummer case when $\ell \mid \deg{F}$.
We can detect the curves $C_F$ with $F \in \mathcal{F}_{q, \ell}$ using the following lemma.

\begin{lemma}\label{lem:4.3} Let $f \in \F_{q^{n_q}}[t]$. Then, $N_{n_q}(f)$ is square-free iff $f = \mathfrak{p}_1 \cdots \mathfrak{p}_s$ where the $\mathfrak{p}_i$ are such that $N_{n_q}(\mathfrak{p}_i)$ are distinct $n_q$-divisible primes of $\F_q[t]$.
\end{lemma}
\begin{proof} 
If $f = \mathfrak{p}_1 \cdots \mathfrak{p}_s$, where the $\mathfrak{p}_i$ are such that $N_{n_q}(\mathfrak{p}_i)$ are distinct $n_q$-divisible primes of $\F_q[t]$, then it is clear that $N_{n_q}(f)=N_{n_q}( \mathfrak{p}_1)\cdots N_{n_q}( \mathfrak{p}_s)$ is square-free. 

Now assume that $N_{n_q}(f)=N_{n_q}( \mathfrak{p}_1)\cdots N_{n_q}( \mathfrak{p}_s)$ is square-free. Then it is clear that the $N_{n_q}(\mathfrak{p}_i)$ are distinct primes in $ \F_{q}[t]$.  
Finally, they are $n_q$-divisible, since they are the result of taking the $N_{n_q}$-norm.  \end{proof}

\begin{definition}
 For a one-variable polynomial $f(t) \in \overline{\F_q}[t]$, let $f^*(u,v):=v^{\deg(f)} f(u/v)$ denote the homogeneous polynomial in variables $u,v$ resulting from the change of variables $t=u/v$. 
\end{definition}

\begin{lemma} \label{change-of-variable}
Let $F \in \mathcal{F}_{q,\ell}$, with $\mathfrak{F}_{1}\in \mathcal{F}^{(1)}_{q,\ell} $ given by \eqref{decomposition}
and $C_{F}$  given by \eqref{BSM}. As in \eqref{product-of-f}, we write $\mathfrak{F}_{1} = f_1 f_2^2\cdots f_{\ell-1}^{\ell-1}$, where $f_i \in \F_{q^{n_q}}[t]$ are pairwise co-prime and square-free.
\begin{itemize}
\item Let $h(t)$ be a non-constant polynomial in $\F_q[t]$ such that $$N_{n_q}(f_1 (h(t)) f_2(h(t)) \cdots f_{\ell-1}(h(t)))$$ is square-free. 
Then, $(F \circ h)(t) 
= N_{n_q}(\mathfrak{F}_{1}(h(t))) \in \mathcal{F}_{q,\ell}$. Let $C_{F\circ h}$ be given by \eqref{BSM}.
Then, 
\begin{eqnarray*}
C_{F \circ h} &\longrightarrow& C_F \\
\left(t, y\right) &\mapsto & \left( h(t), y \right)
\end{eqnarray*}
is a non-constant  map from $C_{F\circ h}$ to $C_F$. 
\item Assume that $n_q > 1$. Let $u(t), v(t)$ be non-constant polynomials in $\F_q[t]$ such that 
$$N_{n_q} \left( f_1^*(u,v) 
\cdots f_{\ell-1}^*(u,v) \right) $$  
is square-free.
 Then $G(t) = N_{n_q} \left( \mathfrak{F}_{1} ^*(u(t),v(t))   \right) \in \mathcal{F}_{q,\ell}$.
 Let $C_{G}$ be given by \eqref{BSM}.
Then
\begin{eqnarray*}
C_{G} &\longrightarrow& C_F \\
\left( t, y\right) &\mapsto & \left(u(t)/v(t), yv(t)^{-\deg(F_{{\bf v}_0})/\ell}\right)
\end{eqnarray*}
is a non-constant map from $C_{G}$ to $C_F$. 
\item Assume that $n_q = 1$ and write $\deg{F} = A \ell - \delta$, where $0 \leq \delta \leq \ell-1$. Let $u(t), v(t)$ be non-constant polynomials in $\F_q[t]$ such that 
$f_1^*(u,v) f_2^*(u,v) \cdots f_{\ell-1}^*(u,v)$ is square-free. Let  $g_i^* = f_i^*$ for $i \neq \delta$ and $g_\delta^* = v f_\delta^*$. Then, 
$g_1^*(u,v) g_2^*(u,v) \cdots g_{\ell-1}^*(u,v)$ is also square-free and
$G(t) =g_1^*(u,v) g_2^*(u,v)^2 \cdots g_{\ell-1}^*(u,v)^{\ell-1} \in \mathcal{F}_{q, \ell}$. Let $C_G: y^\ell = G(t)$. Then
\begin{eqnarray*}
C_{G} &\longrightarrow& C_F \\
\left( t, y\right) &\mapsto & \left(u(t)/v(t), yv(t)^{-A}\right)
\end{eqnarray*}
is a non-constant map from $C_{G}$ to $C_F$. 
\end{itemize}
\end{lemma}

\begin{proof}

We prove the second and third point in the statement, as the first point is a consequence of them. 
First consider the case where $n_q > 1$. We replace $t$ by  $u(t)/v(t)$ in equation \eqref{BSM} and we get
 \begin{eqnarray*}
\prod_{j=0}^{\ell-1} \left( y - \sum_{k=0}^{n_q-1} \zeta_\ell^{j q^k} \sqrt[\ell]{ \frac{F_{\mathbf{v}_k}^*(u,v)}{v^{\deg(F_{\mathbf{v}_k})}}}\right) = 0.
\end{eqnarray*}
 
Recall from Lemma \ref{lem:dividibility} that for the non-Kummer case, $\ell \mid  \deg(F_{\mathbf{v}_k})$. Notice also that the $\mathbf{v}_k$ are all permutations of each other. In fact, $\mathbf{v}_{k+1}$ can be constructed from $\mathbf{v}_k$ by shifting each element one place to the right cyclically and using the fact that $q^{n_q}\equiv 1 \bmod{\ell}$. Writing $A=\frac{\deg(F_{\mathbf{v}_k})}{\ell}$, and making the change of variables $Y=v^Ay$, we finally have 
\begin{eqnarray*}
\prod_{j=0}^{\ell-1} \left( Y - \sum_{k=0}^{n_q-1} \zeta_\ell^{j q^k} \sqrt[\ell]{F_{\mathbf{v}_k}^*(u,v)}\right) = 0,
\end{eqnarray*}
 which is $C_G$ for $G(t) = N_{n_q} \left( \mathfrak{F}_{1} ^*(u(t),v(t))   \right)$.
 
 We now consider the Kummer case. We replace $t$ by  $u(t)/v(t)$ in $y^\ell = F(t)$ to get
 $$v^{A \ell} y^\ell = v^\delta F^*(u,v) = g_1^*(u,v) g_2^*(u,v)^2 \dots g_{\ell-1}^*(u,v)^{\ell-1},$$ and with the change of variables $Y = v^A y$, we get 
 $$Y^\ell = g_1^*(u,v) g_2^*(u,v)^2 \dots g_{\ell-1}^*(u,v)^{\ell-1},$$
 which is $C_G$ for $G(t) = g_1^*(u,v) g_2^*(u,v)^2 \dots g_{\ell-1}^*(u,v)^{\ell-1}.$
\end{proof}

Then Lemma \ref{change-of-variable} translates the conditions for finding curves $C_G$ with a map to $C_F$ to detecting when 
$N_{n_q} \left( f_1^*(u,v) 
\cdots f_{\ell-1}^*(u,v) \right)$ is square-free. We can now proceed to the proof of  
Proposition \ref{general-counting}.

\begin{proof}[Proof of Proposition \ref{general-counting}] Our proof follows  the argument of \cite{Donepudi-Li}, but without restricting to the particular case where $n_q=1$ and $\ell \mid \deg{F}$. 
We concentrate on the parts of their argument where using the general setting explained above introduces some changes, and we just refer to their article for the parts of their argument that can be directly used. 

Let $F=F_0$ be as in Lemma \ref{change-of-variable} and let $C_0=C_{F_0}$ be the curve \eqref{BSM}. Let $d_0$ be the degree of the conductor. 
We now give a lower bound for the number of  $\ell$-cyclic covers with conductor of degree smaller than $n$ that can be obtained by the process of  Lemma \ref{change-of-variable} applied to $F_0$, by using the square-free sieve over $\F_q[t]$.

Let
\begin{align*} 
\mathcal{P}(n) 
&= \{ (D_1,  \dots,  D_{\ell-1}) \in (\F_{q^{n_q}}[t])^{\ell-1} \;:\; D_1, \dots ,D_{\ell-1} \; \text{pairwise co-prime, monic, square-free,} \\
& \hspace{1cm}  \mathfrak{F}_{1} = D_1 \cdots D_{\ell-1}^{\ell-1} \in \mathcal{F}^{(1)}_{q, \ell},  \;\deg{(D_1 \cdots D_{\ell-1})} \leq n \}\\
&= \{ (D_1,  \dots,  D_{\ell-1}) \in (\F_{q^{n_q}}[t])^{\ell-1} \;:\; D_1, \dots ,D_{\ell-1} \;  \text{monic}, N_{n_q} (D_1 \cdots D_{\ell-1}) \; \text{square-free}, \\
& \hspace{1cm} \deg{(D_1 \cdots D_{\ell-1})} \leq n \},
\end{align*}
where the second line follows from Lemma \ref{lem:4.3}.

By the above discussion, each tuple $(D_1, \dots, D_{\ell-1}) \in \mathcal{P}(n)$ gives rise to the $\ell$-cyclic cover $C_F$ where $\mathfrak{F}_{1}=D_1 D_2^2 \cdots D_{\ell-1}^{\ell-1}$ and $F
= N_{n_q} \left(  \mathfrak{F}_{1} \right)$. The conductor is $N_{n_q}(D_1 \cdots D_{\ell-1})$ of degree $\leq n_q n$, and then the genus  is such that {$g \leq \frac{\ell-1}{2}(n_qn-2)$.}

We write $\mathfrak{F}_{1}^0 = f_1 f_2^2 \cdots f_{\ell-1}^{\ell-1}$
where $f_i \in \F_{q^{n_q}}[t]$ and $N_{n_q}(\mathfrak{F}_{1}^0)=F_0$. Notice that $d_0=\deg (N_{n_q}(f_1\cdots f_{\ell-1}))=n_q(\deg(f_1)+\cdots+\deg(f_{\ell-1}))$.
We count the number of distinct $(D_1, \dots, D_{\ell-1}) \in \mathcal{P}(n)$ such that there exists $(u,v) \in \F_q[t]^2$ with
\begin{equation}\label{eq:Dg}
D_1 (t) = f_1^*(u(t),v(t)), \dots, D_{\ell-1} (t) =f_{\ell-1} ^*(u(t),v(t)).
\end{equation}

We then need to detect when $N_{n_q}(D_1 \cdots D_{\ell-1})$ is square-free. Let $G(u, v)$ denote the homogeneous polynomial such that
$$N_{n_q} ( 
f_1^*(u,v) \cdots f_{\ell-1}^*(u,v)) = G(u,v).$$

\kommentar{
\ccom{I think this covers also the Kummer case, if this is true that $f_1^*(u,v) \cdots f_{\ell-1}^*(u,v))$ square-free implies that $g_1^*(u,v) \cdots g_{\ell-1}^*(u,v))$ also square-free. I am not sure why I am worried about that, as you guys seem to find it obvious, but if you are OK, I am}
\mcom{The reason this is obvious lies in the process of homogenization. When you go from $f(t)$ to $v^{\deg(f)}f(u/v)$ you just put enough powers of $v$ to make $f(u/v)$ a polynomial in the variables $u,v$ and not extra power at all. To have that $v$ divides $f^*(u,v)$ you would need to multiply by extra powers, but since you only multiply by $v^{\deg(f)}$ you have no room for extra powers. In other words, if $f(t)=a_0+\cdots +a_n t^n$, then $f^*(u,v)=a_0v^n+\cdots +a_nu^n$ and you're always guaranteed to have the term $a_nu^n$ which is not a multiple of $v$.}}

We now apply a result of Poonen \cite{Poonen} which counts the number of square-free values of $G(u,v)$ as $u, v$ runs over polynomials in $\F_q[t]$, as given in \cite{Donepudi-Li} 
in a form suitable for our application.
\begin{proposition}\cite[Theorem 8.1]{Poonen} \cite[Proposition 3.4]{Donepudi-Li} \label{Poonen-sieve} 	Let $P$ be a finite set of primes in $\mathbb{F}_q[t]$, $B$ be the localization of $\mathbb{F}_q[t]$ by inverting the primes in $P$, $K = \mathbb{F}_q(t)$, $f \in B[x_1, \ldots , x_m]$ be a polynomial that is square-free as an element of $K[x_1, \ldots , x_m]$ and for a choice of $x \in \mathbb{F}_q[t]^m$, we say that $f(x)$ is square-free in $B$ if the ideal $(f(x))$ is a product of distinct primes in $B$. For $b \in B$, define $|b| = |B/(b)|$ and for $b=(b_1,\ldots, b_n) \in B^n$, define $ |b| = \max{|b_i|}$. Let
	\begin{align*}
	S_f &:= \{ x \in \mathbb{F}_q[t]^m : f(x) \text{ is square-free in } B \},\\
	\mu_{S_f} &:= \lim\limits_{N \rightarrow \infty} \frac{|\{ b \in S_f : |b|<N \}|}{N^m}. 
	\end{align*}
	For each nonzero prime $ \pi$ of $B $, let $c_\pi$ be the number of $ x \in ( A / \pi^2 )^m $ that satisfy $f(x) = 0$ in $  A / \pi^2 $. 
	The limit  $\mu_{S_f}$ exists and is equal to $ \prod_{\pi}(1-c_\pi / |\pi|^{2m})$.
\end{proposition}

We then apply Proposition \ref{Poonen-sieve} to $G(u,v)$.
Following \cite[Remark 3.5]{Donepudi-Li}, let $B$ be the localization of $\F_{q}[t]$ by the set of primes $\pi$ with $|\pi|\leq \deg(N_{n_q}(f_1\cdots f_{\ell-1}))= d_0$. This guarantees that
$$
\mu_{S_G} = \lim_{N\rightarrow \infty} \frac{|\{b\in \F_{q}[t]^2, \, |b| \leq N : \, G(b) \mbox{ is square-free in } B \}|}{N^2} >0.
$$

The curve $C_F$ associated to $F = N_{n_q}(D_1 D_2^2 \cdots D_{\ell-1}^{\ell-1})=F_0^*(u(t),v(t))$ as in \eqref{eq:Dg} has genus bounded by $\frac{\ell-1}{2}(d_0\deg(u(t)/v(t))-2)$, and therefore, 
 if we want to guarantee that the genus of $C_F$ is less or equal than $g$, we can prescribe that 
\begin{equation}\label{eq:degbound}
\deg(u(t)/v(t)):=\max \{\deg u(t), \deg v(t)\} \leq \frac{g+\ell-1}{g_0+\ell-1},
\end{equation}
where $g_0$ is the genus of $C_{F_0}$.
 
Now we want to give an upper bound for the $b=(u,v) \in \F_{q}[t]^2$ satisfying condition \eqref{eq:degbound} such that equation \eqref{eq:Dg} is satisfied. Now take $N=q^{n}$, with $n=\frac{2g}{\ell-1}+2$, and we impose the condition $\max\{\deg u, \deg v\}\leq n/d_0$. Notice that 
\[\frac{n}{d_0}=\frac{2g+2(\ell-1)}{d_0(\ell-1)}=\frac{2g+2(\ell-1)}{(d_0-2)(\ell-1)+2(\ell-1)}=\frac{g+\ell-1}{g_0+\ell-1},\]
and therefore condition \eqref{eq:degbound} is satisfied. Applying Proposition \ref{Poonen-sieve}, we get a positive proportion of $\gg \mu N^{2/d_0}=\mu q^{2n/d_0}$ such that  $N_{n_q}(D_1 \cdots D_{\ell-1})$ is square-free. 

To conclude, for a fixed tuple $(D_1, \dots, D_{\ell-1})$ we need to find an upper bound on the number of pairs $(u(t),v(t))$ such that \eqref{eq:Dg} is satisfied in order to correct a double counting. Following a similar reasoning to \cite{Donepudi-Li}, we bound this number by $qn^2q^{\varepsilon n}$. 

In total, for $n$ sufficiently large,   we have
\[\gg \mu q^{  n(2/d_0-\varepsilon)}\]
elements in $\mathcal{P}(n)$ corresponding to $\ell$-cyclic covers of $\mathbb{P}^1_{\F_q}$ with conductor of degree bounded by $n$ that admit a non-constant map to $C_0$. 
\end{proof}

We then need a geometric condition for the vanishing of $\cL(C, u)$ at some point $u=u_0^{-1}$, where $C$ is a curve over $\F_q$. 
This is given by the following theorem of Li
  \cite[Section 2]{Li-vanishing} relating the existence of a rational map between curves to the divisibility of the $L$-functions. The 
 proof uses Honda-Tate theory, which states that every $q$-Weil number is an eigenvalue of the geometric Frobenius acting on the $\ell$-adic Tate module of a simple abelian variety over $\F_q$, which is unique up to isogeny.
 We refer the reader to \cite[Section 2]{Li-vanishing} for the details, and the proof of the following theorem.
 
 \begin{theorem} \label{thm-Li} Let $u_0$ be a $q$-Weil number and let $A_0$ be (the isogeny class of) the unique simple Abelian variety over $\F_q$ having $u_0$ as a Frobenius eigenvalue, as guaranteed by the theorem of Honda--Tate.
 Let $C$ be a curve over $\F_q$. Then, $\cL(C, u_0^{-1})=0$ if and only if there exists a non-trivial map $C \rightarrow A_0$ if and only if $\cL(A_0, u)$ divides $\cL(C, u)$.
 \end{theorem}
 
 \begin{proof} [Proof of Theorems \ref{thm1} and \ref{thm2}] The proof of Theorem \ref{thm1} follows directly from Proposition \ref{prop-4.1} and Theorem \ref{thm-Li}: let $C_0$ be the $\ell$-cyclic cover associated to $\chi_0$, i.e. $\cL(C_0, u_0^{-1})=0$.  By Proposition \ref{prop-4.1} and Theorem \ref{thm-Li}, there are at least $q^{2n/d_0}$ $\ell$-cyclic covers with conductor of degree $\leq n$ such that
 $\cL(C_0, u) \mid \cL(C, u) = \prod_{i=1}^{\ell-1} \cL(\chi^i, u)$, and then at least $q^{2n/d_0}$ characters of order $\ell$ and conductor of degree $\leq n$ such that $\cL(\chi, u_0^{-1})=0$.
 
 The proof of Theorem \ref{thm2} follows directly from Corollary \ref{coro:Li} and the above. Indeed, 
 if $E=E_0 \times_{\F_q} \F_q(t)$ and there exists $\chi_0$ such that $\cL(E, \chi_0, q^{-1})=0$, then by Corollary \ref{coro:Li},  $\cL(C_{\chi_0}, \alpha_1^{-1}) = 0$, and we reason as above.
 \end{proof}

\subsection{Explicit equation for $\ell$-cyclic covers} \label{explicit-BSM}

We now give more information about the equation \eqref{BSM}, including a precise formula for $n_q=2$,
 using the work of Gupta and Zagier \cite{GuptaZagier}.
We used these general formulas for $n_q=2$ to obtain the equations for the curves $C_1, C_2$ and $C_3$ in Section \ref{section-numerical-constant}.

Let $\ell$ be an odd prime number coprime to $q$, let $\omega_\ell$ denote a complex $\ell$-root of unity, and let $\mathcal{R}_{\ell, q}$ denote a set of coset representatives of $(\Z/\ell\Z)^*$ modulo the cyclic subgroup  $\langle q \rangle$. Following  \cite{GuptaZagier}, we define the  polynomial the complex polynomial 
\begin{equation}\label{eq:psi}
\Psi_{\ell,n_q}(y)= \prod_{j \in \mathcal{R}_{\ell, q}} \left( y - \sum_{k=0}^{n_q-1} \omega_\ell^{j q^k} \right),
\end{equation}
This is a polynomial of degree $\frac{\ell-1}{n_q}$. Notice that for $n_q=1$, $\Psi_{\ell,1}(y)$ gives the $\ell$th cyclotomic polynomial and for $n_q=2$,  $\Psi_{\ell,2}(y)$ gives the $\ell$th real cyclotomic polynomial.

Gupta and Zagier prove various results regarding the coefficients of $\Psi_{\ell,n_q}(y)$, and in particular, they recover a formula of Gauss:
\begin{equation}\label{eq:psi2}
\Psi_{\ell,2}(y)=\sum_{n=0}^\frac{\ell-1}{2} (-1)^{\left \lfloor \frac{\ell-1-2n}{4}\right\rfloor} \binom{\left \lfloor \frac{\ell-1+2n}{4}\right\rfloor}{n}y^n.
\end{equation}
In the following result we relate the coefficients in the equation defining $C_F$ in \eqref{BSM} to those of $\Psi_{\ell,n_q}$. Together with the results of \cite{GuptaZagier}, and \eqref{eq:psi2} in particular, this allows us to compute a more explicit formula for equation \eqref{BSM} in the case $n_q=2$.

\begin{proposition} Let  $\ell$ be an odd prime coprime to $q$ and let 
$\Psi_{\ell,n_q}(y)$ be defined as in \eqref{eq:psi}. 
Let $a_m$ be the coefficients of the following polynomial
\begin{equation}\label{am}
y^\ell+\sum_{m=0}^{\ell-1} a_{m} y^{m}:=\Psi_{\ell,n_q}(y)^{n_q}(y-n_q).
\end{equation}

Then, $a_m \in \Z$, and there exists certain coefficients $b_{s_0,\dots,s_{n_q-1}}\in \F_p\subseteq \F_q$ such that the equation defining $C_F$ in \eqref{BSM} can be written as 
\begin{equation}\label{bwaaaa}
C_F:y^\ell+\sum_{m=0}^{\ell-1} \sum_{\substack{0\leq s_k\\\ \sum_{k=0}^{n_q-1}s_k=\ell -m \\\sum_{k=0}^{n_q-1} q^ks_k\equiv 0 \bmod{\ell} }} b_{s_0,\dots,s_{n_q-1}} \mathfrak{F}_1^{\frac{1}{\ell}\sum_{k=0}^{n_q-1} s_k [q^k]_\ell} \mathfrak{F}_2^{\frac{1}{\ell}\sum_{k=0}^{n_q-1} s_k [q^{k-1}]_\ell} \cdots \mathfrak{F}_{n_q}^{\frac{1}{\ell}\sum_{k=0}^{n_q-1} s_k [q^{k+1-n_q}]_\ell}  y^{m}=0.
\end{equation}
Furthermore,  the  $b_{s_0,\dots,s_{n_q-1}} $ satisfy
\begin{equation}\label{ab}
\sum_{\substack{0\leq s_k\\\ \sum_{k=0}^{n_q-1}s_k=\ell-m\\\sum_{k=0}^{n_q-1} q^ks_k\equiv 0 \bmod{\ell}}} b_{s_0,\dots,s_{n_q-1}}=a_m,
\end{equation}
where the $a_m$ are given by \eqref{am} and the equality takes place in $\F_p\subseteq \F_q$ after reducing the $a_m$ modulo $p$ (the characteristic of $\F_q$).

In particular, for $n_q=2$, we have 
\begin{equation}\label{nice}
C_F: y^\ell+\sum_{r=1}^{\frac{\ell-1}{2}} a_{2r-1} (\mathfrak{F}_1\mathfrak{F}_2)^{\frac{\ell+1}{2}-r} y^{2r-1}-\mathfrak{F}_1\mathfrak{F}_2(\mathfrak{F}_1^{\ell-2}+\mathfrak{F}_2^{\ell-2})=0.
\end{equation}
\end{proposition}
Before proceeding to the proof, we remark that the condition $\sum_{k=0}^{n_q-1} q^ks_k\equiv 0 \bmod{\ell}$  implies that  $\sum_{k=0}^{n_q-1} q^{k-j} s_k\equiv 0 \bmod{\ell}$ (since $(q,\ell)=1$), and therefore each of the exponents of the $\mathfrak{F}_j$ in \eqref{bwaaaa} is an integer. One can also see that the $b_{s_0,\dots,s_{n_q-1}}$ are invariant by cyclic permutation of the subindexes. Each of these cyclic permutations results in a permutation in the exponents of the   $\mathfrak{F}_j$. Thus, the final polynomial is symmetric in the $\mathfrak{F}_j$.

\begin{proof}
The initial step of the proof follows  from the elementary fact that 
\[\Psi_{\ell,n_q}(y)^{n_q}(y-n_q)= \prod_{j =0}^{\ell-1} \left( y - \sum_{k=0}^{n_q-1} \omega_\ell^{j q^k} \right).\]
Since the above polynomial has coefficients in the  algebraic integers $\overline{\Z}$, and is invariant under Galois action, we conclude that $\Psi_{\ell,n_q}(y)^{n_q}(y-n_q)\in \Z[y]$ and $a_m \in \Z$. 

Following some ideas from \cite{GuptaZagier}, we consider more generally
\[f_{\ell,n_q} (A_0,\dots,A_{n_q-1})=\prod_{j=0}^{\ell-1} \left( 1 -\sum_{k=0}^{n_q-1}\omega_\ell^{j q^k} A_k\right),\]
and we remark again that this polynomial has coefficients in $\Z$. 

Taking the formal  logarithm,
\begin{align*}
-\log f_{\ell,n_q}(A_0,\dots,A_{n_q-1}) =&\sum_{j=0}^{\ell-1} \sum_{m=1}^\infty \frac{\left(\sum_{k=0}^{n_q-1}\omega_\ell^{j q^k} A_m\right)^m}{m}\\
=&\sum_{j=0}^{\ell-1} \sum_{m=1}^\infty \frac{1}{m}\sum_{\substack{h_0+\cdots +h_{n_q-1}=m\\h_i\geq 0}} \binom{m}{h_0,\dots,h_{n_q-1}} \omega_\ell^{\sum_{k=0}^{n_q-1} j q^kh_k}  A_0^{h_1} \cdots A_{n_q-1}^{h_{n_q-1}}
\\
=& \sum_{m=1}^\infty \frac{1}{m}\sum_{\substack{h_0+\cdots +h_{n_q-1}=m\\h_i\geq 0}} \binom{m}{h_1,\dots,h_{n_q}} A_0^{h_0} \cdots A_{n_q-1}^{h_{n_q-1}}\sum_{j=0}^{\ell-1} \omega_\ell^{j\sum_{k=0}^{n_q-1} q^kh_k} 
\end{align*}
and the innermost sum is zero unless $\sum_{k=0}^{n_q-1} q^kh_k\equiv 0 \bmod{\ell}$. 

In conclusion, the only powers of $A_0,\dots,A_{n_q-1}$ appearing in the Taylor series of  $\log f_{\ell,n_q}(A_0,\dots,A_{n_q-1})$ and consequently in the Taylor series of $f_{\ell,n_q}(A_0,\dots,A_{n_q-1})$ are of the form $A_0^{s_0}\cdots A_{n_q-1}^{s_{n_q-1}}$ 
such that \begin{equation}\label{argh}\sum_{k=0}^{n_q-1} q^ks_k\equiv 0 \bmod{\ell}.\end{equation} But the total degree of $f_{\ell,n_q}$ is $\ell$, and therefore $0 \leq s_0+\cdots + s_{n_q-1}\leq \ell$.
Putting this information together, we obtain 
\begin{equation}\label{bwa}
f_{\ell,n_q} (A_0,\dots,A_{n_q-1})=1+\sum_{m=0}^{\ell-1} \sum_{\substack{0\leq s_k\\\ \sum_{k=0}^{n_q-1}s_k=\ell -m \\\sum_{k=0}^{n_q-1} q^ks_k\equiv 0 \bmod{\ell} }} b_{s_0,\dots,s_{n_q-1}} A_0^{s_0}\cdots A_{n_q-1}^{s_{n_q-1}}.
\end{equation}
Reducing modulo $p$ (the characteristic of $\F_q$), making the change of variables \[A_k=\frac{\sqrt[\ell]{F_{{\bf v}_k}}}{y}=\frac{1}{y}\mathfrak{F}_1^{\frac{[q^k]_\ell}{\ell}}\mathfrak{F}_2^{\frac{[q^{k-1}]_\ell}{\ell}}\cdots \mathfrak{F}_{n_q}^{\frac{[q^{k+1-n_q}]_\ell}{\ell}},\]
and multiplying by $y^\ell$, we obtain equation \eqref{bwaaaa}. Identity \eqref{ab} follows from comparing with \eqref{am}.

When $n_q=2$, we have $q\equiv -1\bmod{\ell}$. Equation \eqref{argh} and condition $\sum_{k=0}^{n_q-1} s_k = \ell - m$ reduce the choices of $s_0,s_1$ to two cases: either $s_0=s_1$ and $m\not = 0$ or $(s_0,s_1)=(0,\ell), (\ell,0)$ and $m=0$. 

For the case $s_0=s_1$, we can set $A_0=A_1$ and reduce to the case of \cite[Theorem 3]{GuptaZagier} to find the coefficients of each $(A_0A_1)^{s_1}$. We then replace $A_0=\frac{\sqrt[\ell]{\mathfrak{F}_1\mathfrak{F}_2^{\ell-1}}}{y}$, $A_1=\frac{\sqrt[\ell]{\mathfrak{F}_1^{\ell-1}\mathfrak{F}_2}}{y}$ (or equivalently, we replace $A_0A_1$ by $\frac{\mathfrak{F}_1\mathfrak{F}_2}{y}$), and obtain the coefficients $a_m$ for $m\not =0$ from the statement. In this case one can see from working with $\Psi_{\ell, 2}(y)$ that $a_m=0$ for $m$ even different from 0. 

The cases $(s_0,s_1)=(0,\ell), (\ell,0)$ only occur for the constant coefficient in \eqref{bwaaaa}
which is
\begin{align*}
(-1)^\ell \omega_\ell^{0+\cdots +(\ell-1)} (A_0^\ell +A_1^\ell)=-(A_0^\ell +A_1^\ell).
\end{align*}
Replacing again $A_0=\frac{\sqrt[\ell]{\mathfrak{F}_1\mathfrak{F}_2^{\ell-1}}}{y}$, $A_1=\frac{\sqrt[\ell]{\mathfrak{F}_1^{\ell-1}\mathfrak{F}_2}}{y}$ and multiplying by $y^\ell$ gives equation \eqref{nice}.
\end{proof}

\section{Numerical data} 

\subsection{Description of the code} \label{section-algo}

We want to compute $L$-functions $\cL(E, \chi, u)$ described by \eqref{L-E-chi}, where $\chi$ is a character of conductor $F$. To simplify, we are choosing $q=p$ to be prime.

Following Section \ref{section-2}, the $L$-functions are polynomials of degree $\degL = \deg{N_E} + 2 \deg{F} - 4 + 2 \delta_\chi$, and 
\[ \cL(E,\chi,u)=\sum_{n=0}^\degL \left( \sum_{f\in\mathcal{M}_n} a_f \chi(f)\right) u^n = \sum_{n=0}^\degL c_n u^n,\]
where $\mathcal{M}_n$ is the set of monic polynomials of degree $n$ in $\F_p[t]$.

Using the functional equation \eqref{FE-E-chi}, we get
\begin{equation} \label{recurrence}
c_n = \omega_{E \otimes \chi} \; p^{2(n-\lfloor \degL/2 \rfloor -1)} \;\overline{c_{\degL-n}}, \;\; 0 \leq n \leq \degL,
\end{equation}
and it suffices to compute $c_i$ for $0 \leq i \leq \lfloor \degL/2 \rfloor.$ \footnote{It follows from \eqref{recurrence} that we can compute numerically the sign of the functional equation by computing $c_{\degL/2}$ when $\degL$ is even, and $c_{\lfloor \degL/2 \rfloor}$ and 
$c_{\lfloor \degL/2 \rfloor + 1}$ when $\degL$ is odd. We used this in the numerical data to compute twists of the Legendre curve by odd characters, as in this case Theorem \ref{thm-sign-FE}
 does not apply. Of course, this requires $c_{\degL/2} \neq 0$. When $c_{\degL/2} = 0$, we computed the next coefficient $c_{(\degL/2)+1}$ to get the sign of the functional equation. In all the cases considered,
 $c_{(\degL/2)+1}$ was not zero (when $c_{\degL/2} =0$), so this was enough.}
 
We then need to compute the $a_f$ appearing in \eqref{L-E-chi}, for $\deg{f} \leq \degL/2$. It follows from the Euler product that
$a_{fg}=a_f a_g$ for $(f,g)=1$, and for $P \in \F_q[t]$ and $n \geq 1$, 
\begin{align*}
a_{P^n} &=
\begin{cases}
a_P a_{P^{n-1}}-pa_{P^{n-2}}, & \text{if } P \nmid N_E,\\
a_P a_{P^{n-1}}, & \text{if } P\mid N_E,
  \end{cases} 
\end{align*}
where $p$ is the characteristic of $\F_q$.

We now turn to the computation of the $a_P$ of a fixed curve $E: y^2 = x^3 + a(t) x^2 + b(t) x + c(t)$.
For $P$ prime, we compute $a_P$ using
\[ a_P = - \sum_{\substack{x\in\mathbb{F}_p[t]\\ \deg(x)<\deg(P)}} \left( \frac{x^3+a(t)x^2+b(t)x+c(t)}{P}\right).\]

After we have computed all $a_f$ for $\deg{f} \leq (\deg{N_E} + 2 d -4 + 2 \delta_\chi)/2$, we can evaluate $\cL(E, \chi, u)$ for any Dirichlet character with conductor of degree $d$ over $\F_p[t]$.
We go through the characters of order $\ell$ and conductor degree $d$ in the following way. Let $n_p$ be the multiplicative order of $p$ modulo $\ell$ as before. Let $F\in\mathbb{F}_p[t]$ be a polynomial of degree $d$ supported on $n_p$-divisible primes. We can enumerate all characters of order $\ell$ and conductor $F$ by choosing only one character per cyclic extension of order $\ell$ of $\F_q(t)$, since the $L$-functions of the $\ell-1$ characters associated to the same extension $K$ vanish together. Writing $F=P_1\cdots P_k$, where the $P_i$ are distinct $n_p$-divisible primes, 
and 
$P_i = \mathfrak{P}_{i,1} \cdots \mathfrak{P}_{i,n_p}$ 
over $\mathbb{F}_{p^{n_p}}(t)$, we consider the (non-conjugate) characters of conductor $F$ over $\mathbb{F}_p(t)$ given by 
\begin{equation} \label{running-over-chi}
\chi(A) = \chi_{{\mathfrak{P}}_{1,1}} (A) \prod_{j=2}^k \chi_{{\mathfrak{P}_{j,1}}}^{a_j} (A), \end{equation} for $a_j \in \{1,\dots,\ell-1\}$, and 
where each $\chi_{{\mathfrak{P}}_{j,1}}$ is the $\ell$th-power residue symbol modulo ${{\mathfrak{P}}_{j,1}}$ over  $\mathbb{F}_{p^{n_p}}(t)$ defined in Section \ref{section-2}.

\subsection{Vanishing of twists of constant curves: numerical data} \label{section-numerical-constant} \label{section-isotrivial}

Let $E_0$ be an elliptic curve over $\F_p$ with $\cL(E_0, u) = (1 - \alpha_0 u) (1-\overline{\alpha}_0 u)$, and let $E = E_0 \times_{\F_p} \F_p(t)$. By \eqref{Artin-conjecture-1},  $\cL(E, \chi, p^{-1})=0$ for some character $\chi$ associated to $K/\F_q(t)$  if and only if $\cL(E/K, p^{-1})=0$, and using the
results of Section \ref{section-constant}, this is equivalent to 
\begin{equation*} 
\cL(E_0, u) \mid \cL(C_\chi, u) = \prod_{j=1}^{\ell-1} \cL(\chi^j, u).\end{equation*}

By Theorem \ref{isotrivial-counting}, once we have found one $\chi_0$ such that $\cL(C_{\chi_0}, \alpha_0^{-1})=0$, then there are infinitely many, so we concentrate on finding $\chi_0$. We examined degree $2$ factors of $\cL(\chi^j, u)$ which arise as $ \cL(E_0, u)$ for some $E_0$ over $\F_p$.

In particular, we considered the case where $\cL(\chi, u)$ has degree 2, which in the case of even (respectively odd) characters means that the conductor of $\chi$ is a polynomial of degree 4 (respectively 3) in $\F_q[t]$. Table \ref{table:d4} presents results for this case: for fixed values of $\ell$ and $p$, we computed $\cL(\chi, u)$ for all characters such that
$\cL(\chi, u)$ is a polynomial of degree 2, and we listed all the cases that we found where $\cL(\chi, u) = \cL(E_0, u)$ for some elliptic curve $E_0/\F_p$. Notice that this means 
$\cL(C_\chi, u) = \cL(E_0, u)^{\ell-1}.$
Each entry in Table \ref{table:d4} may correspond to many characters $\chi$. We did not count them, but our program keeps an instance for each case. For example, the curve $C_1/\F_5$ given by
$$y^3 + (2t^4+2t^3+t^2+4t+4) y + (3t^6+2t^5+2t^4+2t^3+t^2+t+3) = 0$$
has $L$-function $\cL(C_1, u) = (1+5u^2)^2$;
the curve $C_2/\F_{59}$ given by 
\begin{align*}
& y^5 + (54t^4+18t^3+34t^2+18t+39)y^3 + (5t^8+23t^7+44t^6+20t^5+35t^4+30t^3+17t^2+33t+21)y \\ & +(57t^{10}+18t^9+24t^8+58t^7+14t^6+9t^5+41t^4+17t^3+38t^2+48t+44) =0
\end{align*}
has $L$-function $\cL(C_2, u) = (1+59u^2)^4$; and the curve $C_3/\F_{13}$ given by
\begin{align*}
 &y^7 + (6t^4+6t^3+6t^2+12t+1)y^5 + (t^8 + 2t^7+3t^6+6t^5+t^4+5t+4)y^3 + \\
 &(6t^{12}+5t^{11}+10t^{10}+7t^8+2t^7+3t^6+9t^5+3t^4+2t^3+6t^2+t+4)y + \\ &(11t^{14}+6t^{13}+12t^{12}+10t^{11}+5t^{10}+8t^9+6t^8+2t^7+2t^6+10t^5+7t^4+12t^3+3t^2+3t+9) = 0
\end{align*}
has $L$-function $\cL(C_3, u) = (1+13u^2)^6.$

Of course, it would be interesting to prove some criteria which guarantees the existence of a character of degree $\ell$ over $\F_p$ such that $\cL(E_0, u)$ divides $\cL(\chi, u)$. From the data, we are led to believe that this could always be the case when $n_p=2$ and $\cL(E_0, u) = 1 + pu^2$, corresponding to the isogeny class of supersingular elliptic curves over $\F_p$, but we currently do not have a proof. 
We present further evidence  for larger values of $\ell$ in Table \ref{table:conjecture}. Since this becomes more time-consuming, we only consider a thin family of the characters of order $\ell$, where $a_j=1$ for all $j$ in \eqref{running-over-chi}. In some cases ($(\ell,p)=(13,103), (17,101), (31,61)$, and $(37,73)$), we did not go over all characters in the thin family, we stopped after we found $\cL(\chi, u)=(1+p u^2)$, so there might be other characters where $\cL(\chi, u) = (1+a_pu+pu^2).$
In summary, the following is true for all the cases that we tested: for every $\ell, p$ such that $n_p=2$, there exists a character $\chi$ of order $\ell$ over $\F_p$  such that $\cL(\chi, u) = 1 + pu^2$. 

\begin{remark}
	There is a large amount of work in the literature on Newton polygons of cyclic covers of $\mathbb{P}^1$, in particular on the existence of supersingular and superspecial curves. See for example, \cite{LMPT1,LMPT2,LMPT3}. But the existence of the curves we present in this paper does not follow from previous work. In fact, the existence of supersingular curves in families of cyclic covers which ramify at $4$ points with growing degree $\ell$ is surprising from a dimension counting perspective. More surprisingly, these curves are defined over the prime field $\F_p$.
\end{remark}

\begin{center}
\begin{table}[h]
\renewcommand\arraystretch{1.3}
\begin{tabular}{||c||c||c||c||} 
 \hline
 \hline
$\ell$   & $p$    & $n_p$   &$\mathcal{L}(\chi, u)= 1+a_pu+pu^2$    \\
 \hline
 \hline
 \multirow{6}*{3}
  & 5 & 2    &  $0,3$ \\
 \cline{2-4}
  & 7 & 1    &  $-2,-1,1,2,4$ \\
 \cline{2-4}
  & 11 & 2    &  $-3,0,3,6$ \\
 \cline{2-4}
  & 13 & 1    &  $-5,-4,-2,-1,1,2,4,5$ \\
 \cline{2-4}
  & 17 & 2    &  $-6,-3,0,3,6$ \\
 \cline{2-4}
  & 19 & 1    &  $-8,-7,-5,-4,-2,-1,1,2,4,5,7,8$ \\
 \hline
 \hline

 \multirow{7}*{5} 

        & 3 & 4    &  $\varnothing$ \\
    \cline{2-4}

    & 7 & 4    &  $3$ \\
\cline{2-4}

    & 11 & 1    &  $-2,2,3$ \\
\cline{2-4}

    & 13 & 4    &  $-1,4$ \\
\cline{2-4}

    & 19 & 2    &  $0,5$ \\
\cline{2-4}

    & 29 & 2    &  $0$ \\
\cline{2-4}

    & 31 & 1    &  $-2,2,3,8$ \\
    \hline \hline 

 \multirow{2}*{7} 

        & 13 & 2   &  $0$ \\
    \cline{2-4}

    & 29 & 1    &  $-2,2,5$ \\
\cline{2-4}

    \hline \hline 
\multirow{2}*{11} 
& 23 & 1   &  $\varnothing $ \\
    \cline{2-4}
    & 43 & 2    &  $0$ \\
\cline{2-4}
\hline \hline
\multirow{1}*{13} 
& 5 & 4   &  $\varnothing $ \\
\hline \hline 
\multirow{1}*{61} 
& 11 & 4   &  $\varnothing $ \\
\hline \hline 
\end{tabular} 
\bigskip
\caption{All instances of $E_0$ for which there is a $\chi$ of order $\ell$ over $\F_p$ such that $\cL(\chi, u) = \cL(E_0, u)$ for some elliptic curve $E_0 / \F_p$.} \label{table:d4}
\end{table}
\end{center}

\begin{center}
\begin{table}[h]
\renewcommand\arraystretch{1.3}
\begin{tabular}{||c||c||c||c||} 
 \hline
 \hline
$\ell$   & $p$    & $n_p$   &$\mathcal{L}(\chi, u)= 1+a_pu+pu^2$    \\
 \hline
 \multirow{1}*{13}
  & 103 & 2    &  $0$ \\
 \hline \hline 
 \multirow{2}*{17}
  & 67 & 2    &  $0$ \\
  \cline{2-4} 
 & 101 & 2    &  $0$ \\
  \hline \hline 
  \multirow{1}*{19}
  & 37 & 2    &  $0$ \\
 \hline \hline 
  \multirow{1}*{31}
  & 61 & 2    &  $0$ \\
 \hline \hline 
  \multirow{1}*{37}
  & 73 & 2    &  $0$\\
 \hline \hline 
\end{tabular}
\bigskip
\caption{More cases where there is a character $\chi$ of order $\ell$ over $\F_p$ such that $\cL(\chi, u) = (1+p^2u)$. For the cases $(\ell,p) = (17,67)$ and $(19,37)$, we considered all characters in the thin family, and we did not find any other cases where $\cL(\chi, u) = \cL(E_0, u)$ except for $\cL(E_0, u) = (1+p^2 u)$. For the other cases, we stopped after finding $\chi$ such that  $\cL(\chi, u) = (1+p^2u)$, and we did not find any other $\cL(E_0, u)$ up to that point.}\label{table:conjecture}
\end{table}
\end{center}

\subsection{Vanishing of twists of non-constant curves: numerical data} \label{section-numerical}

We now present data for the vanishing of 
 $\cL(E, \chi, p^{-1})$, where $\chi$ varies over characters of order $\ell$ over the finite field $\F_p$ for some prime $p$, and where $E$ is a non-constant curve.
 We used the Legendre curve $E_1: y^2=x(x-1)(x-t)$ and the curve $E_2: y^2 = (x-1)(x-2t^2-1)(x-t^2)$. 
 
 We remark that $E_1$ has conductor $N_1 = t(t-1)P_\infty^2$,
 discriminant $\Delta_1=16t^2(t-1)$, and $j$-invariant $j_1=\frac{256(t^2-t+1)^3}{t^2(t-1)^2}$. Thus, it is smooth and non-constant and has bad reduction at $P_\infty$. 
 Since $\deg(N_1)=4$, we conclude that $\mathcal{L}(E_1,u)=1$. Since the algebraic rank is bounded by the analytic rank (see \cite{Tate}) and this last one equals 0, we conclude that  $E_1$ has (algebraic) rank 0 over $\F_q(t)$. 
 
 Similarly, $E_2$ has conductor $N_2 = t(t-1)(t+1)(t^2+1)$, discriminant $\Delta_2=64t^4(t-1)^2(t+1)^2(t^2+1)^2$, and $j$-invariant $j_2=\frac{1728(t^4+1)^3}{t^4(t-1)^2(t+1)^2(t^2+1)^2}$. Thus, it is smooth and non-constant and has good reduction at $P_\infty$.  Since $\deg(N_2)=5$, we have $\mathcal{L}(E_2,u)=1\pm qu$, and the rank of $E_2$ over $\F_q(t)$ is at most 1.  Let $i$ be a primitive four root of unity in $\overline{\F}_q$, and consider the point 
 $$P=((1+i)t^2+(1+i)t+1,(-1+i)t(t+1)(t-i))$$
 in $E_2(K)$, where $K = \F_q(t)(i)$. One can see that the N\'eron--Tate height of $P$ is positive, and therefore $P$ has infinite order (see  the book of Shioda and Sch\"utt \cite{SS-book} for a general reference).  As before, we use that  the algebraic rank is bounded by the analytic rank \cite{Tate}. 
 If $q \equiv 1 \mod 4$, then $K=\F_q(t)$, and we conclude that  $E_2$ has (algebraic) rank exactly 1 over $\F_q(t)$. Therefore $\cL(E_2, u) = 1 - qu$. 
If $q \equiv 3 \mod 4$, then $K=\F_{q^2}(t)$, and $K/\F_q(t)$ is a quadratic constant field extension. Therefore $\cL(E/K, u) = 1 - q^2 u$, since $\deg N_E - 4 = 1$.
We also have \begin{equation}
\label{quadratic-twist} \mathcal{L}(E_2/K, {u^2}) = \mathcal{L}(E_2, u) \cL(-E_2, u),\end{equation}
where 
 \[-E_2: -y^2= (x-1)(x-2t^2-1)(x-t^2).\]
 We remark that we have $\cL(E_2, u^2)$ and not $\cL(E_2, u)$ in \eqref{quadratic-twist} because $K/\F_q(t)$ is a constant field extension (see \cite[Chapter 8]{Rosen} for more details).  When $q \equiv 3 \mod 4$, the point $2P=(t^2+1,it^2)$ defined over $\F_{q^2}(t)$ yields a (non-torsion) point $\tilde{P}=(t^2+1,t^2)$ defined over $\F_q(t)$ on  $-E_2$. Thus the algebraic rank of $-E_2$  over $\F_q(t)$ is $1$ and 
 $\cL(-E_2, u) = 1 - qu$. Now \eqref{quadratic-twist} implies that $\cL(E_2, u) = 1 + qu$.
   In conclusion, we have that \[\cL(E_2,u)=\begin{cases}1-qu \quad \mbox{ if } \quad q \equiv 1 \bmod{4},\\ 1+qu  \quad \mbox{ if } \quad q \equiv 3 \bmod{4}.\end{cases}\]

We present in Tables \ref{Legendre-3}, \ref{Legendre-5}, and \ref{Legendre-7} our results for twists of the Legendre curve  with characters of order 3, 5, and 7 respectively, and various ground fields $\F_p(t)$.  For the  curve given by  $y^2=(x-1)(x-2t^2-1)(x-t^2)$, we present in Tables \ref{David-3}, \ref{David-5}, and \ref{David-7} our results for twists of this curve with characters of order 3, 5, and 7 respectively, and various ground fields $\F_p(t)$.
We have also tested higher order twists ($\ell=11,13$ for $E_1$ and $\ell=11,31,71$ for $E_2$) but without finding any vanishing. This data is presented in Tables \ref{higher-Legendre} and \ref{higher-David}.

Each table has the same format: the first three columns are the values of $\ell$, $p$ and $n_p$ and the fourth column is the degree $d$ of the conductors of the characters of order $\ell$ over $\F_p(t)$ considered (then, $n_p$ always divides $d$). The $L$-functions $\cL(E, \chi, u)$ are then computed for all $\chi$ of order $\ell$ over $\F_p(t)$ with conductor of degree $d$, and they are classified according to their analytic rank, which is defined as $\text{rank}(\chi) = r_{\text{an}}(E, \chi) = \mbox{ord}_{u=q^{-1}} \cL(E, \chi, u)$. Since $\text{rank}(\chi^i)=\text{rank}(\chi^j)$, we only count one power per character in our data. Then, the next 
columns give the number of such $\chi$ where the analytic rank is 0, or 1, or $2, \dots$
The most extensive computation that we did was for twists of order $\ell=3$ of the curve $E_2$ for conductors of degree 8 over $\F_5(t)$, where we needed to compute $a_P$ for primes of degree $\leq 8$, which is the most involved part of computing the twisted $L$-functions $\cL(E_2, \chi, u)$ for characters with conductors of degree 8. This took approximately 20 days on an Intel(R) Core(TM) i5-4300U CPU.
This is also the only case where we found a twist  of analytic rank 3.

The data for the Legendre  curve is very compatible with the conjectures of \cite{DFK} and \cite{MR-experimental}, as we have found no instances of vanishing for any character of order 7 or higher. For the curve $E_2$, we have found many instances of vanishing for characters of order 7, but none for characters of higher order.

\begin{center}
\begin{table}[h] 
\renewcommand\arraystretch{1.3}
\begin{tabular}{||c||c||c||c||c|c|c||}
 \hline
 \hline
twist order   & $p$  &$n_p$  & $\deg$ conductor $d$  & rank 0   & rank 1 & rank 2    \\
 \hline
 \hline
 \multirow{8}*{3}
            & \multirow{4}*{5}& \multirow{4}*{2}
                        & 2    &  6   & 4   & 0 \\
 \cline{4-7}
            &          & &4 & 205    & 32 & 3    \\
 \cline{4-7}
            &        &   &  6         & 5784    & 260 & 16    \\
 \cline{4-7}
            &        &   &  8         & 302640    & 116 & 4    \\
\cline{2-7}
  & \multirow{4}*{7}& \multirow{4}*{1}
                        & 1    &  5   & 0   & 0 \\
                         \cline{4-7}
        &       &  & 2    &  37   & 4   & 0 \\
         \cline{4-7}
         &     &  & 3    &  324   & 37   & 1 \\
          \cline{4-7}
         &   &      & 4    &  2935   & 73   & 0 \\
         \hline \hline 
\end{tabular}
\bigskip
\caption{Twists of order  3 for the Legendre curve.}\label{Legendre-3}
   \end{table}
   \end{center}

\begin{center}
\renewcommand\arraystretch{1.3}
\begin{table}[h] 
\begin{tabular}{||c||c||c||c||c|c||}
 \hline
 \hline
twist order   & $p$ & $n_p$   & $\deg$ conductor $d$    & rank 0   & rank 1    \\
 \hline
 \hline
 \multirow{6}*{5}
            & \multirow{1}*{7} & \multirow{1}*{4}
                        & 4    &  585   & 3    \\
                        \cline{2-6}
                      & \multirow{4}*{11} & \multirow{4}*{1}
                       & 1  &  9   & 0    \\
                        \cline{4-6}
                      & & & 2   &  199   & 0    \\
                        \cline{4-6}
                       && & 3   &  3759   & 5    \\
                        \cline{4-6}
                    & & & 4    &  65143   & 11    \\
                        \cline{2-6}   
                         & \multirow{1}*{19}& \multirow{1}*{2}
                        & 2    &  170   & 1   \\
                        \hline\hline
                        \end{tabular}
\bigskip
\caption{Twists of order  5 for the Legendre curve.}  \label{Legendre-5}
\end{table}
\end{center}    

\begin{center}
\renewcommand\arraystretch{1.3}
\begin{table}[h] 
\begin{tabular}{||c||c||c||c||c||}
 \hline
 \hline
twist order   & $p$  &$n_p$  & $\deg$ conductor $d$    & rank 0      \\
 \hline
 \hline
 \multirow{12}*{7}
            & \multirow{1}*{5} & \multirow{1}*{6}
                        &  6   &  2580  \\
                        \cline{2-5}
                          & \multirow{1}*{11} & \multirow{1}*{3}
                        &  3   &  440  \\
                        \cline{2-5}
                          & \multirow{2}*{13} & \multirow{2}*{2}
                        &  2   &  78  \\
                          \cline{4-5}
                     &  & &  4   &  25116 \\
                        \cline{2-5}
                         & \multirow{1}*{23} & \multirow{1}*{3}
                        &  3  &  4048  \\
                        \cline{2-5}
                & \multirow{3}*{29} & \multirow{3}*{1}
                        &  1  &  27  \\
                        \cline{4-5}
                        &&&2& 2512\\
                        \cline{4-5}
                        &&&3&179192\\
                        \cline{2-5}
                      & \multirow{1}*{41} & \multirow{1}*{2}
                        &  2  &  820  \\
                        \cline{2-5}
                         & \multirow{1}*{197} & \multirow{1}*{1}
                        &  1  &  195  \\
                        \cline{2-5}
                         & \multirow{1}*{337} & \multirow{1}*{1}
                        &  1  &  335  \\
                        \cline{2-5}
                         & \multirow{1}*{379}  & \multirow{1}*{1}
                        &  1  &  377  \\
                        \hline \hline
                        \end{tabular}
\bigskip
\caption{Twists of order  7 for the Legendre curve. We have found no instances of vanishing in this case.}  \label{Legendre-7}
\end{table}
\end{center}

\begin{center}
\renewcommand\arraystretch{1.3}
\begin{table}[h] 
\begin{tabular}{||c||c||c||c||c||}
 \hline
 \hline
twist order   & $p$  &$n_p$   & $\deg$ conductor $d$    & rank 0      \\
 \hline
 \hline
 \multirow{5}*{11}
            & \multirow{1}*{5} &  \multirow{1}*{5}
                        &  5   &  624  \\
                        \cline{2-5}
                          & \multirow{1}*{23}  & \multirow{1}*{1}
                        &  1   &  21  \\
                        \cline{2-5}
                          & \multirow{1}*{43}  & \multirow{1}*{2}
                        &  2   &  903  \\
                          \cline{2-5}
                         & \multirow{1}*{67}  & \multirow{1}*{1}
                        &  1  &  65  \\
                        \cline{2-5}
                & \multirow{1}*{89}  & \multirow{1}*{1}
                        &  1  &  87  \\
\hline \hline
 \multirow{4}*{13}
            & \multirow{1}*{5}  & \multirow{1}*{4}
                        &  4   &  150  \\
                         \cline{2-5}
                          & \multirow{1}*{29}  & \multirow{1}*{3}
                        &  3   &  8120  \\
                        \cline{2-5}
                           & \multirow{2}*{53}  & \multirow{2}*{1}
                        &  1   &  51  \\
                        \cline{4-5}
                   &   &  &  2   &  16678  \\
                        \hline \hline
                        \end{tabular}
\bigskip
\caption{Twists of order  11 and 13 for the Legendre curve. We have found no instances of vanishing in this case.}  \label{higher-Legendre}
\end{table}
\end{center}

 \vfill\eject

 \begin{center}
\renewcommand\arraystretch{1.3}
\begin{table}[h]
\begin{tabular}{||c||c||c||c||c|c|c|c||}
 \hline
 \hline
twist order   & $p$  &$n_p$  & $\deg$ conductor $d$    & rank 0   & rank 1 & rank 2  &rank 3  \\
 \hline
 \hline
 \multirow{24}*{3}
            & \multirow{4}*{5} & \multirow{4}*{2}
                        & 2    &  8   & 2   & 0  & 0 \\
 \cline{4-8}    
         &  & & 4    &  214   & 26   & 0  & 0\\
          \cline{4-8}    
         &  & &   6  &  5780   & 280   & 0  & 0\\
          \cline{4-8}    
                &  & &   8  &  149222   & 2136   & 20 & 2 \\
         \cline{2-8}
          & \multirow{6}*{7}& \multirow{6}*{1}
                        & 1    &  4   & 0   & 0  & 0\\
                         \cline{4-8}    
         &  & & 2    &  30   & 2   & 0  & 0\\ \cline{4-8}    
         & &  & 3    &  264   & 22   & 2  & 0\\ \cline{4-8}    
         &  & & 4    &  2299   & 49   & 4  & 0\\ \cline{4-8}    
         &  & & 5   &  18670   & 240   & 2  & 0\\ \cline{4-8}    
         & &  & 6    &  148537   & 1343   & 32  & 0 \\ \cline{2-8} 
           & \multirow{1}*{11} & \multirow{1}*{2}
                        & 2   &  53   & 0   & 1  & 0\\
                         \cline{2-8}  
                         & \multirow{3}*{13} & \multirow{3}*{1}
                        & 1   &  8   & 0   & 0  & 0\\
                         \cline{4-8} 
                         &   &  & 2   &  122   & 12   & 0 & 0 \\
                         \cline{4-8} 
                         &   &  & 3   &  2140   & 56   & 4  & 0\\
                         \cline{2-8}
                         & \multirow{1}*{17} & \multirow{1}*{2}
                        & 2   &  116   & 20   & 0  & 0\\
                        \cline{2-8}
                         & \multirow{2}*{19} & \multirow{2}*{1}
                        & 1   &  14  & 2   & 0  & 0\\
                         \cline{4-8} 
                         &  &   & 2   &  380   & 28   & 2  & 0\\
                         \cline{2-8}
                          & \multirow{1}*{23} & \multirow{1}*{2}
                        & 2   &  244   & 6   & 2  & 0\\
                        \cline{2-8}
                         & \multirow{1}*{29} & \multirow{1}*{2}
                        & 2   &  364   & 42   & 0  & 0\\
                        \cline{2-8}
                         & \multirow{2}*{31} & \multirow{2}*{1}
                        & 1   &  26   & 2   & 0  & 0\\
                        \cline{4-8}
                      &  & & 2   &  1190   & 24   & 6  & 0\\
                       \cline{2-8}
                         & \multirow{1}*{103} & \multirow{1}*{1}
                        & 1   &  100   & 0   & 0  & 0\\
                        \cline{2-8}
                         & \multirow{1}*{109} & \multirow{1}*{1}
                        & 1   &  104   & 0   & 0  & 0\\
                         \cline{2-8}
                         & \multirow{1}*{151} & \multirow{1}*{1}
                        & 1   &  146   & 2   & 0  & 0 \\
   \hline \hline 
\end{tabular}
\bigskip
\caption{Twists of order  3 for the curve  $y^2=(x-1)(x-2t^2-1)(x-t^2)$.} \label{David-3}
   \end{table}
    \end{center}

   \begin{center}
\renewcommand\arraystretch{1.3}
\begin{table}[h]
\begin{tabular}{||c||c||c||c||c|c|c||}
 \hline
 \hline
twist order   & $p$ & $n_p$   & $\deg$ conductor $d$    & rank 0   & rank 1 & rank 2    \\
 \hline
 \hline
 \multirow{11}*{5}
            & \multirow{1}*{7} & \multirow{1}*{4}
                        & 4    &  587   & 0   & 1 \\
 \cline{2-7} 
   & \multirow{3}*{11} & \multirow{3}*{1}
                        & 1    &  8   & 0   & 0 \\
                        \cline{4-7}
                        & & & 2    &  166   & 0   & 0 \\
                        \cline{4-7}
                         & & & 3    &  3064   & 0   & 0 \\
 \cline{2-7}
 & \multirow{1}*{19}& \multirow{1}*{2}
                        & 2    &  170   & 0   & 0 \\
                        \cline{2-7}
 & \multirow{1}*{29} & \multirow{1}*{2}
                        & 2    &  388   & 18   & 0 \\
 \cline{2-7} 
 & \multirow{2}*{31} & \multirow{2}*{1}
                        & 1    &  28   & 0   & 0 \\
                        \cline{4-7}
                  &    &  & 2   &  1975   & 0   & 1 \\
 \cline{2-7} 
  & \multirow{1}*{41} & \multirow{1}*{1}
                        & 1    &  36   & 0   & 0 \\
                         \cline{2-7} 
  & \multirow{1}*{101} & \multirow{1}*{1}
                        & 1    &  96   & 0   & 0 \\
                          \cline{2-7} 
  & \multirow{1}*{131} & \multirow{1}*{1}
                        & 1    &  128   & 0   & 0 \\
   \hline \hline 
\end{tabular}
\bigskip
\caption{Twists of order  5 for the curve  $y^2=(x-1)(x-2t^2-1)(x-t^2)$.}  \label{David-5}
   \end{table}
    \end{center}

    \begin{center}
\renewcommand\arraystretch{1.3}
\begin{table}[h]
\begin{tabular}{||c||c||c||c||c|c||}
 \hline
 \hline
twist order   & $p$ &$n_p$    & $\deg$ conductor $d$    & rank 0   & rank 1    \\
 \hline
 \hline
 \multirow{7}*{7}
            & \multirow{1}*{5} & \multirow{1}*{6}
                        & 6    &  2560   &20    \\
 \cline{2-6} 
  & \multirow{1}*{11} & \multirow{1}*{3}
                        & 3   &  440   &0  \\
 \cline{2-6} 
 & \multirow{2}*{13} & \multirow{2}*{2}
                        & 2    &  72   & 6    \\
                        \cline{4-6}
                    &    &  & 4    &  24984   & 132    \\
 \cline{2-6} 
 & \multirow{2}*{29}& \multirow{2}*{1}
                        & 1    &  24   & 0   \\
                        \cline{4-6}
                  &    &  & 2   &  2046   & 16   \\
 \cline{2-6} 
 & \multirow{1}*{41} & \multirow{1}*{2}
                        & 2    &  800   & 20   \\
   \hline \hline 
\end{tabular}
\bigskip
\caption{Twists of order  7 for the curve  $y^2=(x-1)(x-2t^2-1)(x-t^2)$.}  \label{David-7}
   \end{table}
    \end{center}   
    \clearpage
    \begin{center}
\renewcommand\arraystretch{1.3}
\begin{table}[h] 
\begin{tabular}{||c||c||c||c||c||}
 \hline
 \hline
twist order   & $p$ &$n_p$   & $\deg$ conductor $d$    & rank 0      \\
 \hline
 \hline
 \multirow{9}*{11}
            & \multirow{1}*{5}  & \multirow{1}*{5}
                        &  5   &  624  \\
                        \cline{2-5}
                          & \multirow{3}*{23}   & \multirow{3}*{1}
                        &  1   &  20  \\
                         \cline{4-5}
                        &&&2 & 2152\\
                         \cline{4-5}
                        &&&3&168448\\
                        \cline{2-5}
                          & \multirow{1}*{43} & \multirow{1}*{2}
                        &  2   &  902  \\
                          \cline{2-5}
                         & \multirow{2}*{67}  & \multirow{2}*{1}
                        &  1  &  64  \\
                        \cline{4-5}
                      &  &&2 & 22370\\
                        \cline{2-5}
                & \multirow{1}*{89}  & \multirow{1}*{1}
                        &  1  &  84  \\
                        \cline{2-5}
                      & \multirow{1}*{199}  & \multirow{1}*{1}
                        &  1  &  196  \\
                        \hline 
                        \hline 
                        \multirow{1}*{31}  
            & \multirow{1}*{5} & \multirow{1}*{3}
                        &  3   &  40  \\
\hline \hline
         \multirow{1}*{71}
            & \multirow{1}*{5} & \multirow{1}*{5}
                        &  5   &  624  \\
\hline \hline
\end{tabular}
\bigskip
\caption{Twists of order 11, 31, and 71 for the curve $y^2=(x-1)(x-2t^2-1)(x-t^2)$. We have found no instances of vanishing in this case.}  \label{higher-David}
\end{table}
\end{center}



\bibliographystyle{amsalpha}

\bibliography{Bibliography}

\end{document}